\documentclass[12pt]{article}

\usepackage{fullpage}

\normalfont

\usepackage{amsmath}
\usepackage{amsthm,amssymb}
\usepackage{amsfonts}
\usepackage{txfonts}
\usepackage{graphics}
\usepackage{graphicx}
\usepackage[all]{xy}
\usepackage[usenames]{color}

\numberwithin{equation}{section}

\newtheorem{theorem}{Theorem}
\newtheorem{lemma}[theorem]{Lemma}
\newtheorem{proposition}[theorem]{Proposition}

\theoremstyle{definition}

\def\cC{\mathcal{C}}

\def\cF{\mathcal{F}}

\def\cN{\mathcal{N}}

\def\cR{\mathcal{R}}

\def\cX{\mathcal{X}}
\def\cV{\mathcal{V}}

\def\bbR{\mathbb{R}}

\def\vol{\operatorname{Vol}}

\def\eps{\varepsilon}

\newcommand{\pr}[1]{\mathbb{P}\left[#1\right]}
\newcommand{\expect}[1]{\mathbb{E}\left[#1\right]}
\DeclareMathOperator{\dist}{dist}

\def\dM{M}
\def\dG{G}

\begin{document}

\title{\bf The Normalized Graph Cut and Cheeger Constant: from Discrete to Continuous}

\author{
Ery Arias-Castro\footnote{Department of Mathematics, University of California, San Diego, USA},
Bruno Pelletier\footnote{D\'epartement de Math\'ematiques, IRMAR -- UMR CNRS 6625, Universit\'e Rennes II, France}
and Pierre Pudlo\footnote{D\'epartement de Math\'ematiques, I3M -- UMR CNRS 5149, 
Universit\'e Montpellier II, France}
}

\date{\today}

\maketitle

\noindent {\bf Abstract.}
Let $\dM$ be a bounded domain of $\bbR^d$ with smooth boundary.  We relate the Cheeger constant of $\dM$ and the conductance of a neighborhood graph defined on a random sample from $\dM$.  By restricting the minimization defining the latter over a particular class of subsets, we obtain consistency (after normalization) as the sample size increases, and show that any minimizing sequence of subsets has a subsequence converging to a Cheeger set of $\dM$. \\


\vspace{0.2cm}

\section{Introduction and main results}

The Cheeger isoperimetric constant may be defined for a Euclidean domain as well as for a graph.  In either case it quantifies how well the set can be bisected or `cut' into two pieces that are as little connected as possible.  Motivated by recent developments in spectral clustering and computational geometry, we relate the Cheeger constant of a neighborhood graph defined on a sample from a domain and the Cheeger constant of the domain itself.

Given a graph $\dG$ with weights $\{\delta_{ij}\}$, the {\it normalized cut of} a subset $S \subset \dG$ is defined as
\begin{equation}\label{eq:hgraph}
h(S; \dG) = \frac{\sigma(S)}{\min\{\delta(S), \delta(S^c)\}},
\end{equation}
where $S^c$ denotes the complement of $S$ in $\dG$, and 
\begin{equation}\label{eq:deltasigma}
\delta(S) = \sum_{i \in S} \sum_{j \neq i} \delta_{ij}, \quad \sigma(S) = \sum_{i \in S} \sum_{j \in S^c} \delta_{ij},
\end{equation}
are the discrete volume and perimeter of $S$.
The {\it Cheeger constant} or {\it conductance} of the graph $\dG$ is defined as the value of the optimal normalized cut over all non-empty subsets of $\dG$, i.e.
\begin{equation} \label{eq:hG}
H(\dG) = \min\{h(S; \dG): S \subset \dG, S \neq \emptyset\}.
\end{equation}
A corresponding quantity can be defined for a domain of a Euclidean space.
Let $\dM$ be a bounded domain (i.e.~open, connected subset) of $\mathbb{R}^d$ with smooth boundary $\partial\dM$.
For an integer $1 \leq k \leq d$, let $\vol_{k}$ denote the $k$-dimensional volume (Hausdorff measure) in $\bbR^d$.  For an open subset $A \subset \bbR^d$, define its normalized cut with respect to $\dM$ by
$$
h(A; \dM) = \frac{\vol_{d-1}(\partial A \cap \dM)}{\min\{\vol_d(A \cap \dM), \vol_d(A^c \cap \dM)\}},
$$
where $A^c$ denotes the complement of $A$ in $\bbR^d$ and with the convention that $0/0 = \infty$.
The Cheeger (isoperimetric) constant of $\dM$ is defined as
$$
H(\dM) = \inf\{h(A; \dM): A \subset \dM \}.
$$
Equivalently, the infimum may be restricted to all open subsets $A$ of $\dM$ such that $\partial A \cap \dM$ is a smooth submanifold of co-dimension 1.
This quantity was introduced by Cheeger~\cite{MR0402831} in order to bound the eigengap of the spectrum of the Laplacian on a manifold.  A Cheeger set is a subset $A \subset \dM$ such that $h(A; \dM) = H(\dM)$; there is always a Cheeger set and it is unique under some conditions on the domain $\dM$~\cite{caselles2009some}.  For $A \subset \dM$, we call $\partial A \cap \dM$ its relative boundary. 

\subsection{Consistency of the normalized cut}

Suppose that we observe an i.i.d.~random sample $\cX_n = (X_1, \dots, X_n)$ from the uniform distribution $\mu$ on $\dM$.
For $r > 0$, let $\dG_{n,r}$ be the graph with nodes the sample points and edge weights $\delta_{ij} = {\bf 1}\{\|X_i -X_j\| \leq r\}$, which is an instance of a random geometric graph~\cite{penrose}.
Let $\omega_d$ denote the $d$-volume of the unit $d$-dimensional ball, and define 
\begin{equation}
\label{eq:gammav1}
\gamma_{d} = \int_{\bbR^d} \max\big(\langle u,z\rangle, 0\big) \, {\bf 1}\{\|z\| \leq 1\} \, {\rm d}z,
\end{equation}
where $u$ is any unit-norm vector of $\mathbb{R}^d$. 
Actually $\gamma_{d}$ is the average volume of a spherical cap when the height is chosen uniformly at random.
We establish the pointwise consistency of the normalized cut, which yields an asymptotic upper bound on the Cheeger constant of the neighborhood graph based on the Cheeger constant of the manifold.
This is the first result we know of that relates these two quantities.

\begin{theorem}
\label{theo:pointwise}
Let $A$ be a fixed subset of $M$ with smooth relative boundary.
Fix a sequence $r_n \to 0$ with $n r_n^{d+1}/\log n \to +\infty$, and let $S_n = A\cap G_{n,r_n}$.
Then with probability one
$$\frac{\omega_d}{\gamma_{d} r_n}h(S_n;G_{n,r_n}) \to h(A;M),$$
and, consequently,
$$\limsup_{n \to \infty} \frac{\omega_d}{\gamma_{d} r_n} H(\dG_{n, r_n}) \leq H(\dM).$$
\end{theorem}

We do not know whether the Cheeger constant of the neighborhood graph, for an appropriate choice of the connectivity radius and properly normalized, converges to the Cheeger constant of the domain.

\subsection{Consistent estimation of the Cheeger constant and Cheeger sets}

We obtain a consistent estimator of the Cheeger constant $H(\dM)$ by restricting the minimization defining the conductance of the neighborhood graph \eqref{eq:hG} to subsets associated with subsets of $\bbR^d$ with controlled reach.
The reach of a subset $S \subset \bbR^d$~\cite{MR0110078}, denoted ${\rm reach}(S)$, is the supremum over $\eta > 0$ such that, for each $x$ within distance $\eta$ of $S$, there is a unique point in $S$ that is closest to $x$.
We assume here that $M \subset (0,1)^d$.  When this is not known and/or not the case, we may always infer a hypercube that contains $M$---by taking a hypercube containing all the data points, with some lee-way so that the hypercube contains $M$ with high probability when the sample gets large---and then rescale and translate the points so that $M$ is within the unit hypercube.  So this assumption is really without loss of generality.
  
\begin{theorem}
\label{theo:estim}
Assume that $M \subset (0,1)^d$ and that $r_n \to 0$ such that $n r_n^{2d+1} \to \infty$.  Let $\rho_n \to 0$ slowly so that $r_n = o(\rho_n^\alpha)$ and $n r_n^{2d+1} \rho_n^\alpha \to \infty$ for all $\alpha > 0$.
Let $\cR_n$ be a class of open subsets $R \subset (0,1)^d$ such that ${\rm reach}(\partial R) \geq \rho_n$.
Define the functional $h_n^\ddag$ over $\cR_n$ by
\[
h_n^\ddag(R) =
\frac{\omega_d}{\gamma_{d} r_n} h\left(R \cap \cX_n; \dG_{n, r_n}\right)\]
if both $R$ and $R^c$ contain a ball of radius $\rho_n$ centered at a sample point, and $h_n^\ddag(R)=\infty$ otherwise.
\begin{itemize}
\item[$(i)$] With probability one,
$$
\min_{R \in \cR_n} h_n^\ddag(R) \to H(\dM), \quad n \to \infty.
$$
\item[$(ii)$] Let $\{R_n\}$ be a sequence satisfying 
\begin{equation}\label{eq:rrrn}
R_n \in \cR_n, \quad \quad h_n^\ddag(R_n) = \min\{h_n^\ddag(R): R \in \cR_n\}.
\end{equation}
Then with probability one, $\{R_n\cap M\}$ admits a subsequence converging in the $L^1$-metric. Moreover, any subsequence of $\{R_n\cap M\}$ converging in the $L^1$-metric converges to a Cheeger set of $M$.
\end{itemize}
\end{theorem}

Note that the infimum defining $R_n$ in \eqref{eq:rrrn} is attained in $\cR_n$ since the function $h_n^\ddag$ takes only a finite number of values.

Part \textit{(ii)} of Theorem~\ref{theo:estim} hints at a consistent estimate of a Cheeger set of $M$, but $R_n\cap M$ depends on $M$, which is unknown.
On the other hand reconstructing an unknown set from a random sample of it is an independent problem for which there exists multiple techniques and an important literature---see e.g., \cite{bcpsupport} and the references therein.
In the following result we construct a random discrete measure which does not require the knowledge of $M$, and prove that, seen as a sequence of random measures indexed by the sample size $n$, any accumulation point is the uniform measure on a Cheeger set of $M$.

\begin{theorem}
\label{theo:discretemeasure}
Let $\{R_n\}$ be a sequence as in Theorem~\ref{theo:estim}-$(ii)$, $\{R_{n_k}\}$ a subsequence of $\{R_n\}$ with $R_{n_k}\cap M \to A_\infty$ in $L^1$.
Define the random discrete measure $Q_n = \frac{1}{n}\sum_{i=1}^n\mathbf{1}_{R_n}(X_i)\delta_{X_i}$ and the measure $Q = \mathbf{1}_{A_\infty}(.) \mu$.
Then, that $Q_n$ converges weakly to $Q$ is an event which holds with probability one.
\end{theorem}

As an example of an estimate of a Cheeger set of $M$, one can consider a union of balls of radius $\kappa_n$ centered at the observations falling in $R_n$.
Under appropriate conditions, it is known that this estimate converges in $L^1$; see \cite{bcpsupport}.

Let us mention that with our result, only the ``regular'' part of a Cheeger set can be reconstructed.
Indeed, in dimension $d\geq 8$, the boundary of a Cheeger set is not necessarily regular and may contain parts of codimension greater than 1.

\subsection{Connections to the literature}

Our results relating the respective Cheeger constants of a domain and of a neighborhood graph defined from a sample from the domain are the first of their kind, as far as we know.  The connections to the literature stem from the concept of normalized cut taking a central place in graph partitioning and related methods in clustering; from a recent trend in computational geometry (and topology) aiming at estimating geometrical (and topological) attributes of a set based on a sample; and from the fact that we can use the conductance to bound the mixing time of a random walk on the neighborhood graph.  
 
{\bf Clustering.}
In spectral graph partitioning, the goal is to partition a graph $\dG$ into subgraphs based on the eigenvalues and eigenvectors of the Laplacian~\cite{MR2294342,Chung97}.  It arises as a convex relaxation of the combinatorial search of finding an optimal bisection in terms of the normalized cut.
Given a set of points $X_1, \dots, X_n$ and a dissimilarity measure (or kernel) $\phi$, spectral clustering applies spectral graph partitioning to the graph with nodes the data points and edge weight $\delta_{ij} = \phi(X_i, X_j)$ between $X_i$ and $X_j$~\cite{von2007tutorial}.  For instance, if the points are embedded in a Euclidean space, the kernel $\phi$ is often of the form $\phi(x, y) = \psi(\|x -y\|/\sigma)$, where $\sigma$ is a tuning parameter, and $\psi$ is, e.g., the Gaussian kernel $\psi(t) = \exp(-t^2)$ or the simple kernel $\psi(t) = {\bf 1}_{[0,1]}(t)$~\cite{Ng02, belkin2002laplacian}.  The consistency of spectral methods has been analyzed in this context~\cite{vonLuxburg08,pelletier08,MR2460286,MR2387773,MR2238670}.  In particular,~\cite{LDS_NIPS_06} proves a result similar to our Theorem~\ref{theo:pointwise} in that context.  

About cuts,~\cite{maier2009influence} also proves a result similar to our Theorem~\ref{theo:pointwise} when the separating surface $\partial A$ is an affine hyperplane.  Closer to our Theorem~\ref{theo:estim},~\cite{cut-learn} establishes rates for learning a cut for classification purposes---so the setting there is that of supervised learning, with each sample point $X_i$ associated with a class label $Y_i$.  

{\bf Computational geometry (and topology).}
The Cheeger constant $H(\dM)$, and Cheeger sets, are {\em bona fide} geometric characteristics of the domain $\dM$ that we might want to estimate, following a fast developing line of research around the estimation of some geometric and topological characteristics of sets from a sample, e.g., the number of connected components~\cite{MR2320821}, the intrinsic dimensionality~\cite{levina-bickel} and, more generally, the homology~\cite{MR2383768,MR2476414,MR2506738,MR2121296,MR2460371,MR1876386,chazal2009analysis}; the Minkowski content~\cite{MR2341697}, as well as the perimeter and area (volume)~\cite{MR1641826}.

{\bf Random walks.}
Random geometric graphs are gaining popularity as models for real-life networks.  Some protocols for passing information between nodes amounts to performing a random walk and it is important to bound the time it takes for information to spread to the whole network; see~\cite{1244725} and references therein.  It is well-known that, given a graph $\dG$, a lower bound on $H(\dG)$ may be used to bound the mixing time the random walk on $\dG$.  This is the path taken in~\cite{boyd,1244725} when $\dM$ is the unit hypercube and the graph is $\dG_{r_n, n}$.  However, in both papers the authors reduce the setting to that of a regular grid without rigorous justification, leaving the problem unresolved (in our opinion) even in this particular case.

\subsection{Discussion}

As we saw, there are only a handful of other papers relating cuts in neighborhood graphs and cuts in the corresponding domain from which the points making the neighborhood graph where sampled from.  Our paper is the first one we know of that establishes a relationship between the Cheeger constant (optimal normalized cut) on the neighborhood graph and the Cheeger constant of the domain, and the first one to propose a method that is consistent for the estimation of the latter based on a restricted normalized cut, and also consistent for the estimation of Cheeger sets.  Our results generalize with varying amount of effort to other related settings.  However, we leave important questions behind. 

{\bf Generalizations.}
With some additional work, our results and methodology extend to settings where the kernel (here the simple kernel) is fast decaying and where the data points are sampled from a probability distribution on $\dM$ that has a non-vanishing density with respect to the uniform distribution.
It would also be interesting to consider the setting where $M$ is a $d$-dimensional smooth submanifold embedded in some Euclidean ambient space.  Our arguments seem to carry through using a set of charts for the manifold $M$, as is done in~\cite[Lem.~3.4]{buser82}.

{\bf Refinements.}
Though we focused on sufficient conditions for $r_n$ to enable a consistent estimation of the Cheeger constant of the domain, it may also be of interest to find necessary conditions.  Partial work suggests that $n r_n^d \to \infty$ is necessary, and may be sufficient the divergence to infinity is faster than a sufficiently large power of $\log n$.  The arguments in support of this, however, are substantially different than those we use in the paper, which hinge on Hoeffding's inequality for $U$-statistics.  

{\bf An open problem.}
Whether the normalized Cheeger constants of some sequence of neighborhood graphs converges to the Cheeger constant of the domain is an intriguing question. To paraphrase the question we leave open, is there a sequence $\{r_n\}$ such that, with probability one, 
$$\lim_{n \to \infty} \frac{\omega_d}{\gamma_{d} r_n} H(\dG_{n, r_n}) = H(\dM)?$$ 
A positive answer would establish the consistency of the normalized cut criterion for graph partitioning.  Also, a lower bound on $H(\dG_{n, r_n})$ would provide a lower bound on the eigengap between the first and second eigenvalue of the Laplacian, which in turn may be used to bound the mixing time of the random walk on $\dG_{n, r_n}$, as done in~\cite{boyd,1244725} when $\dM$ is the unit hypercube.


{\bf Consistent estimation in polynomial time.}
Our estimation procedures, though theoretically valid and consistent, are not practical.  It would be interesting to know whether there is a consistent estimator for the Cheeger constant that can be implemented in polynomial-time.  Note that computing the Cheeger constant of a graph is NP-hard (which motivates the use of spectral methods), and even the best polynomial-time approximations we are aware of are not precise enough to allow for consistency~\cite{1033179}.

\subsection{Content}
The rest of the paper is devoted to the proofs of the three theorems.
In Section~\ref{sec:pointwise}, we establish the convergence of the discrete volume and perimeter to their continuous counterparts of a fixed subset of $\dM$ with smooth relative boundary, using Hoeffding's inequality for $U$-statistics~\cite{hoeffding}.
Then, by the lower semi-continuity of the map $A\mapsto  h(A;M)$, we deduce the supremum-limit bound of Theorem~\ref{theo:pointwise}.  
In Section~\ref{sec:unif}, we prove Theorems~\ref{theo:estim} and~\ref{theo:discretemeasure} by utilizing results on empirical $U$-processes~\cite{MR1666908} on the one hand, and compactness properties of the $L^1$-metric~\cite{henrot-pierre} on the other hand.

\subsection{Notation and background}
\label{sub:notation}

The uniform measure on $\dM$ is denoted $\mu$, so that $\mu(A) = \vol_d(A \cap \dM)/\vol_d(\dM)$; and the normalized perimeter is denoted $\nu(A) = \vol_{d-1}(\partial A \cap \dM)/\vol_d(\dM)$.
Let $\tau_M = \vol_d(M)$, and define the discrete volume and perimeters as
\begin{equation} \label{def}
\mu_n(A) = \frac{\tau_M}{\omega_d n (n-1) r_n^d} \delta(A \cap \cX_n; \dG_{n, r_n}), \quad 
\nu_n(A) = \frac{\tau_M}{\gamma_{d} n (n-1) r_n^{d+1}} \sigma(A \cap \cX_n; \dG_{n, r_n}),
\end{equation}
where $\delta$, $\sigma$ are given in \eqref{eq:deltasigma}, $\mathcal X_n$ is the sample, and $G_{n,r_n}$ the neighborhood graph.
Also, define the discrete ratio
$$h_n(A) = \frac{\nu_n(A)}{\min(\mu_n(A), \mu_n(A^c))},$$
and note that
\[h_n(A) = \frac{\omega_d}{\gamma_{d} r_n}h(A\cap\cX_n; G_{n,r_n}),
\]
where $h$ is given in \eqref{eq:hgraph}.
For further reference, we define the volume $\pi_d(\eta)$ of a spherical cap at height $\eta$ by
$$
\pi_d(\eta) = \vol_d\big\{x: \|x\| \leq 1 \text{ and } \langle u,x\rangle \geq \eta\big\},
$$
where $u$ is any unit-norm vector of $\mathbb{R}^d$.
Note that the constant $\gamma_d$ defined in \eqref{eq:gammav1} may be expressed as 
$$\gamma_d = \int_0^1\pi_d(\eta){\rm d}\eta.$$

The reach coincides with the condition number introduced in~\cite{MR2383768} for submanifolds without boundary, and the property ${\rm reach}(\partial A) \geq r$ is equivalent to $A$ and $A^c$ being both $r$-convex~\cite{walther}, in the sense that a ball of radius $r$ rolls freely inside $A$ and $A^c$.  (We say that a ball of radius $r$ rolls freely in $A$ if, for all $p \in \partial A$, there is $x \in A$ such that $p\in \partial B(x,r)$ and $B(x, r) \subset A$.) 
It is well-known that the reach bounds the radius of curvature from below~\cite[Thm.~4.18]{MR0110078}.
In particular, if ${\rm reach}(\partial A) >0$, then $\partial A$ is a smooth submanifold (possibly with boundary).

In the rest of the paper, the generic constant $C$ may vary from line to line,  except when stated explicitly otherwise.

\section{Proof of Theorem~\ref{theo:pointwise}: Consistency of the normalized cut}
\label{sec:pointwise}



For a subset $A$ of $\dM$ and a real number $r>0$, define the symmetric kernel
\begin{equation} \label{eq:phi}
\phi_{A,r}(x,y) = \frac{1}{2}\Big\{\mathbf{1}_A(x)+\mathbf{1}_A(y)\Big\}{\bf 1}\{\|x-y\| \leq r\},
\end{equation}
so that $\mu_n(A)$ may be expressed as the following U-statistic:
\[
\mu_n(A) = \frac{\tau_M}{\omega_d n(n-1) r_n^{d}} \sum_{i \neq j} \phi_{A,r_n}(X_i, X_j).
\]
Similarly, $\nu_n(A)$ may be written as 
\[
\nu_n(A) = \frac{\tau_M}{\gamma_{d} n (n-1) r_n^{d+1}} \sum_{i \neq j} \bar{\phi}_{A,r_n}(X_i, X_j).
\]
with the symmetric kernel
\begin{equation}
\label{eq:barphi}
\bar{\phi}_{A,r}(x,y) = \frac{1}{2}\Big\{\mathbf{1}_A(x)\mathbf{1}_{A^c}(y) + \mathbf{1}_A(y)\mathbf{1}_{A^c}(x)\Big\} {\bf 1}\{\|x-y\| \leq r\}.
\end{equation}
We shall need the following Hoeffding's Inequality for $U$-statistics~\cite{hoeffding}, which is a special case of~\cite[Thm.~4.1.8]{MR1666908}.

\begin{theorem}
\label{lem:hoeff}
Let $\phi$ be a measurable, bounded kernel on $\bbR^d \times \bbR^d$ and let $\{X_k: k \in \mathbb{N}\}$ be i.i.d. random vectors in $\bbR^d$.  Assume that $\expect{\phi(X_1, X_2)} = 0$ and that $b := \|\phi\|_\infty < \infty$, and let $\sigma^2 = {\rm Var}(\phi(X_1, X_2))$.  Then, for all $t > 0$,
$$
\pr{\frac{1}{n(n-1)} \sum_{i \neq j} \phi(X_{i}, X_j) \geq t} \leq \exp\left( - \frac{n t^2}{5 \sigma^2 + 3 b t} \right).
$$
\end{theorem}

To prove Theorem~\ref{theo:pointwise}, we establish the almost-sure convergence of $\mu_n(A)$ to $\mu(A)$ and of $\nu_n(A)$ to $\nu(A)$ for a subset $A \subset \dM$ with smooth relative boundary.
To this aim, we combine upper bounds on bias terms together with exponential inequalities for U-statistics.
The bias terms involve volume bounds which we present next, and integrations over some neighborhoods of the boundary of a regular set, namely tubular neighborhoods or simply tubes, which comes after that.

\subsection{Volume bounds}

For any $r>0$, define
\begin{equation} \label{mr}
M_r = \{x \in M: \dist(x, \partial M) \geq r\}.
\end{equation}
The following two lemmas provide bounds on the volume of the intersection of balls with some subsets of $\dM$.

\begin{lemma}
\label{lem:eta}
Let $R$ be a bounded open subset of $\mathbb{R}^d$ with
${\rm reach}(\partial R) = \rho>0$.
Set $A=R\cap M$.
For any $r < \min\{{\rm reach}(\partial M) ; \rho\}$, any $0\leq\eta\leq1$, and all $p$ in $\partial A \cap M_r$, we have
$$\left|{\rm Vol}_d\left(B(p + \eta r e_p,r) \cap A^c\right) - \pi_d(\eta) r^d \right| \leq 2\omega_{d-1} r^{d+1}/\rho,$$
where $e_p$ denotes the unit normal vector at $p$ pointing inward $A$.
\end{lemma}
\begin{proof}
For ease of notation, set $B = B(p+\eta r e_p,r)$.
Let $(\tilde{e}_1,\dots,\tilde{e}_d)$ be an orthonormal frame at $p$, with $\tilde{e}_d = e_p$.
Denote by $\tilde{x}_1,\dots,\tilde{x}_d$ the local coordinates in this frame, such that $p$ has coordinates 0.
Then $\partial A\cap M$ can be expressed locally as the set of points $\tilde{x}$ such that $\tilde{x}^d = F(\tilde{x}^1,\dots,\tilde{x}^{d-1})$ for some function $F$, and, if we set $\tilde{x}^{(d)}=(\tilde{x}^1,\dots,\tilde{x}^{d-1})$, then
\begin{align*}
\vol_d(B\cap A^c) & =  \int_B\mathbf{1}\{ \tilde{x}^d < F(\tilde{x}^{(d)})\}{\rm d}\tilde{x}\\
& =   \int_B \left[ \mathbf{1}\{ \tilde{x}^d < F(\tilde{x}^{(d)})\}\mathbf{1}\{\tilde{x}^d< 0\} + \mathbf{1}\{ \tilde{x}^d < F(\tilde{x}^{(d)})\}\mathbf{1}\{\tilde{x}^d> 0\} \right]{\rm d}\tilde{x}
\end{align*}
Since
$$\pi_d(\eta) r^d = \int_{B}\mathbf{1}\{\tilde{x}^d < 0\}{\rm d}\tilde{x}$$
it follows that
\begin{align*}
 \left|{\rm vol}_d\left(B_n \cap A^c\right) - \pi_d(\eta) r_n^d \right| 
&\leq  \int_B \left[ \mathbf{1}\{ \tilde{x}^d > F(\tilde{x}^{(d)})\}\mathbf{1}\{\tilde{x}^d< 0\} + \mathbf{1}\{ \tilde{x}^d < F(\tilde{x}^{(d)})\}\mathbf{1}\{\tilde{x}^d> 0\} \right]{\rm d}\tilde{x}  \\
&\leq \int_{B_n} \mathbf{1}\left\{ |\tilde{x}^d| \leq |F(\tilde{x}^{(d)})| \right\} {\rm d}\tilde{x}
\leq 2 \int_{\{\|\tilde{x}^{(d)}\| \leq r\}} |F(\tilde{x}^{(d)})| {\rm d}\tilde{x}^{(d)}.
\end{align*}
Expanding $F$ at 0, we have, for all $\tilde{x}$ with $\|\tilde{x}\|\leq r$,
$$F(\tilde{x}^{(d)}) = \sum_{i,j=1}^{d-1}G_{ij}(\xi)\tilde{x}^i\tilde{x}^j,$$
for some $\xi:=\xi(\tilde{x}^{(d)})$.
Since the reach bounds the principal curvatures by $1/\rho$ \cite{MR0110078}, we have $\sup_{p\in\partial A\cap M_r}\|G(p)\| \leq 1/\rho.$
Then, using the change of variable $u=r\tilde{x}$, we deduce that
\begin{align*}
\left|{\rm vol}_d\left(B(p + \eta r e_p,r) \cap A^c\right) - \pi_d(\eta) r_n^d \right| & \leq  2\omega_{d-1} \sup_{p\in\partial A\cap M }\|G(p)\| r^{d+1}\\
& \leq  2\omega_{d-1} r^{d+1}/\rho. \qedhere
\end{align*}
\end{proof}

\begin{lemma}
\label{lem:ball-vol}
There exists some constant $C>0$ such that, for all $r,\alpha$ satisfying $0<2r\le\alpha\le\operatorname{reach}(\partial \dM)$, and all $x$ in $M$,
\[\vol_d( B(x,\alpha) \cap M_{r}) \geq C\alpha^d.\]
\end{lemma}
\begin{proof}
The main argument is to include a ball of radius $\alpha/4$ into $B(x,\alpha)\cap M_r$. We can proceed the following way. First, because $\rho:=\operatorname{reach}(\partial \dM)>0$, for any $x\in M$ there is $y\in M$ such that $x\in B(y,\rho)\subset M$. Second, since  ${\rm dist}(y,\partial M)\ge \rho$ and $\rho\ge 2r$, we have $y\in M_r$ and $B(y,\rho-r)\subset M_r$. Hence
\[
B(x,\alpha)\cap B(y,\rho-r) \subset B(x,\alpha) \cap M_{r}.
\]
If $y=x$, the result is trival. Otherwise, let $z:=x+(r+\alpha/4)(y-x)/\|y-x\|$ and note that $B(z,\alpha/4)$ is a ball of radius $\alpha/4$ included in $B(x,\alpha)\cap B(y,\rho-r)$.
\end{proof}

\subsection{Integration over tubes}
\label{sub:tubes}

We introduce the notion of tubes and some of their properties; see~\cite{MR2024928} for an extensive treatment.  Let $S$ be a submanifold of $\mathbb{R}^d$.  The {\it tubular neighborhood} of radius $r>0$ about $S$, denoted $\cV(S,r)$, is the set of points $x$ in $\mathbb{R}^d$ for which there exists $s \in S$ with $\|x -s\| < r$ and such that the line joining $x$ and $s$ is orthogonal to $S$ at $s$.  When $S$ is without boundary, $\cV(S,r)$ coincides with the set of points $x$ in $\mathbb{R}^d$ at a distance no more than $r$ from $S$.  If $S$ has boundary, then the tube coincides with the set of points at distance no more than $r$, with the ends removed, corresponding to the points projecting onto $\partial S$.  Assume $S$ is of codimension 1, and oriented, and define $e_p$ as the (unit) normal vector of $S$ at $p \in S$.  When $r < {\rm reach}(S)$, $\cV(S,r)$ admits the following parameterization 
$$
\cV(S,r) = \{x = p+te_p: p \in S, -r\leq t \leq r\}.
$$

Denote by $\mathbb{II}_p$ the second fundamental form of $S$ at $p \in S$.
The infinitesimal change of volume function is defined on $S\times(-r;r)$ by $\vartheta(p,t) = \operatorname{det}(I-t\mathbb{II}_p)$; the dependence of $\vartheta$ on $S$ is omitted.
Given an integrable function $g$ on $\cV(S,r)$, we have:
$$\int_{\cV(S,r)}g(x){\rm d}x = \int_S\int_{-r}^rg(p,t)\vartheta(p,t){\rm d}t \, v_\sigma({\rm d}p),$$
where $v_\sigma$ is the Riemannian volume measure on $S$.

\begin{lemma}
\label{lem:chvol}
Assume $S$ is a submanifold of $\bbR^d$ of codimension 1, with $\rho := {\rm reach}(S) > 0$.  Then, for all $r < \rho$,
$$\sup_{p\in S}\sup_{-r\leq t \leq r} \vartheta(p,t) \leq (1 + r/\rho)^{d-1},$$
and
$$\sup_{p\in S}\sup_{-r\leq t \leq r} |\vartheta'(p,t)| \leq \frac{(d-1) (1 + r/\rho)^{d-1}}{\rho-r}$$
where $\vartheta'$ is the derivative of $\vartheta$ with respect to $t$.
\end{lemma}

\begin{proof}
By~\cite[Thm.~4.18]{MR0110078}, the reach bounds the radius of curvature from below so that the principal curvatures $\kappa^{(1)},\dots,\kappa^{(d-1)}$ (the eigenvalues of the second fundamental form) are everywhere bounded (in absolute value) from above by $1/\rho$.    
Therefore, for $r < \rho$ and $-r\leq t \leq r$,
$$
0 \leq \vartheta(p,t) = \operatorname{det}(I-t\mathbb{II}_p) = \prod_{i=1}^{d-1}\left(1 - \kappa^{(i)}_p t\right) \leq (1 + r/\rho)^{d-1}.
$$

For the derivative of $\vartheta$, we have
$$\frac{\vartheta'(p,t)}{\vartheta(p,t)} = - \sum_{i=1}^{d-1}\frac{\kappa^{(i)}_p}{1-\kappa^{(i)}_p t}.$$
Hence
\[
|\vartheta'(p,t)| \leq \vartheta(p,t) (d-1) \frac{1/\rho}{1-r/\rho} \leq \frac{(d-1) (1 + r/\rho)^{d-1}}{\rho-r}. \qedhere
\] 
\end{proof}

The celebrated Weyl's tube formula~\cite{MR1507388} provides fine estimates for the volume of a tubular region around a smooth submanifold of $\bbR^d$.  We only require a rough upper bound of the right order of magnitude, which we state and prove here.

\begin{lemma}
\label{lem:tube}
For any bounded open subset $R \subset \mathbb{R}^d$ with ${\rm reach}(\partial R) = \rho>0$ and any $0 < r < \rho$,
\[\vol_{d}(\cV(\partial R, r))\le 2^d \vol_{d-1}(\partial R) \, r.\]
\end{lemma}

In particular, Lemma~\ref{lem:tube} implies 
\begin{equation} \label{tube}
\mu\left[\cV(\partial M,r)\right] \leq C r, \quad \forall r < {\rm reach}(\partial M),
\end{equation}
where $C$ is a constant depending only on $M$.

\begin{proof}
Using the uniform bound of the infinitesimal change of volume given in Lemma~\ref{lem:chvol}, we have
\begin{align*}
\vol_d\left[\cV(\partial R, r)\right] & = \int_{\partial R}\int_{-r}^{r} \vartheta(p,u){\rm d}u\, v_\sigma({\rm d}p)\\
& \leq \vol_{d-1}(\partial R) \, 2 r (1 + r/\rho)^{d-1} \leq 2^d \vol_{d-1}(\partial R) \, r. \qedhere
\end{align*}
\end{proof}

\subsection{Bounds on bias terms}
\label{sub:bias}

Recall the definition of $M_r$ in \eqref{mr}.

\begin{lemma}
\label{lem:evol}
Let $\phi_{A,r}$ be defined as in \eqref{eq:phi}.
There exists a constant $C$, depending only on $\dM$, such that, for any $A \subset \dM$ and $r < {\rm reach}(\partial M)$, 
$$\left|\frac{\tau_M}{\omega_d r^d} \expect{\phi_{A,r}(X_1, X_2)} - \mu(A)\right| \leq \mu(A \cap \dM_r^c).$$
\end{lemma}

\begin{proof}
Assume without loss of generality that $\tau_M = 1$. 
We first note that
$$
\expect{\phi_{A,r}(X_1, X_2)} = \expect{\mathbf{1}_A(X_1) {\bf 1}\{\|X_1 -X_2\| \leq r\}}.
$$
We partition $A$ into $A \cap \dM_r$ and $A \cap \dM_r^c$.  By conditioning on $X_1$, we have 
\begin{eqnarray*}
\expect{\mathbf{1}_{A \cap \dM_r}(X_1) {\bf 1}\{\|X_1 -X_2\| \leq r\}} &=& \omega_d r^d \mu(A \cap \dM_r) = \omega_d r^d \mu(A) -\omega_d r^d \mu(A \cap \dM_r^c); \\
\expect{\mathbf{1}_{A \cap \dM_r^c}(X_1) {\bf 1}\{\|X_1 -X_2\| \leq r\}} &\leq& \omega_d r^d \mu(A \cap \dM_r^c).
\end{eqnarray*}
Hence the result.
\end{proof}


\begin{lemma}
\label{lem:eper}
Let $A=R\cap M$, where $R$ is a bounded domain with smooth boundary and ${\rm reach}(\partial R) = \rho > 0$.
Let $\bar{\phi}_{A,r}$ be defined as in \eqref{eq:barphi}.
\begin{itemize}
\item[$(i)$]
There exists a constant $C$, depending only on $M$, such that, for any $A \subset \dM$ and $r <\min\{\rho/2, {\rm reach}(\partial \dM)\}$, 
$$\left|\frac{\tau_M}{\gamma_{d} r^{d+1}}\expect{\bar{\phi}_{A,r}(X_1, X_2)} - \nu(A)\right| \leq C\left(\vol_{d-1}(\partial R \cap \cV(\partial M,r)) + \vol_{d-1}(\partial R\cap M)\frac{r}{\rho}\right).$$
\item[$(ii)$] There exists a constant $C$, depending only on $M$, such that, for any $A\subset M$ and $r< \min\{\rho/2, {\rm reach}(\partial \dM)\}$,
\begin{equation}
\frac{\tau_M}{\gamma_{d} r^{d+1}}\expect{\bar{\phi}_{A,r}(X_1, X_2)} - \frac{\vol_{d-1}(\partial A \cap \dM_r)}{\vol_d(\dM)} \geq -C\nu(A) \frac{r}{\rho}.\label{eq:r-eper}
\end{equation} 
\end{itemize}
\end{lemma}

\begin{proof}
Assume without loss of generality that $\tau_M = 1$. 
Let $S$ denote $\partial R \cap \dM$.
Then
$$
\expect{\bar{\phi}_{A,r}(X_1, X_2)}
=
\mathbb{E}\left[\mathbf{1}_A(X_1)\mathbf{1}_{A^c}(X_2)\mathbf{1}\left\{\|X_1-X_2\|\leq r\right\} \right] 
=  \int_{D} \vol_d\left[B(x,r)\cap A^c\right] \mu({\rm d}x),
$$
where
$$D= \left\{x\in A\,:\,{\rm dist}(x,\partial R)\leq r\right\}.$$
Since $r < \rho$, the projection on $\partial R$ is well-defined on $D$, and any $x$ in $D$ can be written as $x = p + t e_p$, for $p \in \partial R$, and with $e_p$ the unit normal vector of $\partial R$ at $p$ pointing inwards.

We partition $D$ into $D \cap \dM_r$ and $D \cap \dM_r^c$.  
Denote by $S_r$ the projection of $D \cap \dM_r$ on $S$.  
We have
\begin{eqnarray*}
\int_{D \cap \dM_r} \vol_d\left[B(x,r)\cap A^c\right] \, {\rm d}x 
& = & \int_{S_r} \int_{-r}^0 \vol_d\left[B(p+te_p,r) \cap A^c\right] \vartheta(p,t) {\rm d}t \, v_\sigma({\rm d}p)\\
& = & r \int_{S_r} \int_0^1 \vol_d\left[ B(p-\eta re_p,r) \cap  A^c \right]\vartheta(p,r \eta) {\rm d}\eta \, v_\sigma({\rm d}p).
\end{eqnarray*}
Therefore
\begin{align}
& \left|\frac{1}{r^{d+1}}\int_{D \cap \dM_r} \vol_d\left[B(x,r)\cap A^c\right] {\rm d}x - \gamma_{d}\nu(A)\right| \label{1st-term} \\ 
&\qquad\leq \frac{1}{r^d} \int_{S_r} \int_0^1\left| \vol_d\left[ B(p-\eta re_p,r) \cap  A^c \right] - \pi_d(\eta)r^d\right|   \vartheta(p,r \eta) {\rm d}\eta \, v_\sigma({\rm d}p) \notag \\
&\qquad\quad+ \left| \int_{S_r} \int_0^1 \pi_d(\eta) \vartheta(p,r \eta) {\rm d}\eta \, v_\sigma({\rm d}p) - \gamma_{d}\nu(A)\right|. \notag
\end{align}
Lemma~\ref{lem:eta} provides the inequality $\left| \vol_d\left[ B(p-\eta re_p,r) \cap  A^c \right] - \pi_d(\eta)r^d\right| \leq 2\omega_{d-1}r^{d+1}/\rho$, and the first inequality of Lemma~\ref{lem:chvol} states that $\sup_{p\in S}\sup_{-r\leq t \leq t} \vartheta(p,t)\leq (1+r/\rho)^{d-1}$.
Since $r<\rho$, $\sup_{p\in S}\sup_{0\leq\eta\leq 1}\vartheta(p,\eta r) \leq 2^{d-1}$.
Hence, the first term on the right-hand side is bounded by
$$
2\omega_{d-1} (r/\rho) \int_{S_r} \int_0^1  \vartheta(p,r \eta) {\rm d}\eta \, v_\sigma({\rm d}p)
\leq 2^{d}\omega_{d-1} (r/\rho) \vol_{d-1}(S_r).
$$

To bound the second term, a Taylor expansion leads to the relation $\vartheta(p,r\eta) = 1 + \vartheta'(p,r\xi_\eta)r\eta$ for some $0<\xi_\eta<1$.
The second inequality of Lemma~\ref{lem:chvol} states that $\sup_{p\in S}\sup_{-r\leq t \leq r} |\vartheta'(p,t)| \leq (d-1) (1 + r/\rho)^{d-1}/(\rho-r)$
so that $\sup_{p\in S}\sup_{0\leq \eta\leq 1} |\vartheta'(p,r\xi_\eta)|$ is bounded by $(d-1)2^d/\rho$ since $r<\rho$.
Recall that the constant $\gamma_d$ is expressed as $\gamma_d= \int_0^1\pi_d(\eta){\rm d}\eta$.
Then the second term in the right-hand side of \eqref{1st-term} is bounded by
\begin{align*}
&\left| \int_{S_r}\int_0^1\pi_d(\eta) {\rm d}\eta \, v_\sigma({\rm d}p) - \gamma_{d}\nu(A)\right| + r \int_{S_r}\int_0^1\eta\pi_d(\eta)|\vartheta'(p,r\xi_\eta)|{\rm d}\eta \, v_\sigma({\rm d}p)\\
&\qquad\leq \gamma_{d} \left|\vol_{d-1}(S_r) - \vol_{d-1}(S)\right| + (d-1)2^d\gamma_{d} (r/\rho) \vol_{d-1}(S_r)\\
& \qquad \leq \gamma_{d} \vol_{d-1}(S\cap M_r^c) + (d-1)2^d\gamma_{d} (r/\rho) \vol_{d-1}(S_r),
\end{align*}
where we have used the fact that $S\backslash S_r \subset M_r^c$ since $S \cap M_r \subset S_r$.
Collecting terms, the term in \eqref{1st-term} is bounded by
\begin{equation} \notag
\gamma_{d} \vol_{d-1}(S \cap \dM_r^c) + C\frac{r}{\rho} \vol_{d-1}(S_r),
\end{equation}
for some constant $C$ independent of $M$.

For the integral over $D \cap \dM_r^c$, since $D$ is included in the intersection of tubes of radius $r$ about $\partial R$ and $\partial M$, i.e., $D\subset \cV(\partial R,r)\cap\cV(\partial M,r) $, we have
\begin{eqnarray*}
\int_{D \cap \dM_r^c} \vol_d\left[B(x,r)\cap A^c\right] \, {\rm d}x 
& \leq & \int_{\partial R\cap \cV(\partial M,r)} \int_{-r}^0 \vol_d\left[B(p+te_p,r) \cap A^c\right] \vartheta(p,t) {\rm d}t \, v_\sigma({\rm d}p)\\
& = & r \int_{\partial R\cap \cV(\partial M,r)} \int_0^1 \vol_d\left[ B(p-\eta re_p,r) \cap  A^c \right]\vartheta(p,r \eta) {\rm d}\eta \, v_\sigma({\rm d}p) \\
& \leq & 2^{d-1} \omega_d r^{d+1} \vol_{d-1}(\partial R \cap \cV(\partial M,r)),
\end{eqnarray*}
where we have used Lemma~\ref{lem:chvol} again to bound $|\vartheta(p,r\eta)|$ by $(1+r/\rho)^{d-1}\le 2^{d-1}$ in the last inequality.

Combining the two inequalities on the integrals over $D \cap \dM_r$ and $D \cap \dM_r^c$, we obtain that
\begin{align*}
& \left|\frac{1}{\gamma_{d} r^{d+1}}\expect{\bar{\phi}_{A,r}(X_1, X_2)} - \nu(A)\right|\\
&\quad \leq \vol_{d-1}(S \cap \dM_r^c) + C\frac{r}{\rho} \vol_{d-1}(S_r) + 2^{d-1} \omega_d \vol_{d-1}(\partial R \cap \cV(\partial M,r))\\
&\quad \leq C\left(\vol_{d-1}(\partial R \cap \cV(\partial M,r)) + \vol_{d-1}(S)\frac{r}{\rho}\right),
\end{align*}
which proves the first bound stated in Lemma~\ref{lem:eper}.\\

To prove $(ii)$, using the bound on \eqref{1st-term}, we deduce that
\begin{eqnarray*}
\frac{1}{\gamma_{d} r^{d+1}}\expect{\bar{\phi}_{A,r}(X_1, X_2)}  & \geq & \frac{1}{\gamma_{d} r^{d+1}} \int_{D \cap \dM_r} \vol_d\left[B(x,r)\cap A^c\right] {\rm d}x\\
& \geq & \vol_{d-1}(S) - \left[\vol_{d-1}(S\cap M_r^c) + \frac{C}{\gamma_{d}}\frac{r}{\rho}\vol_{d-1}(S_r)\right]\\
& \geq & \vol_{d-1}(S\cap M_r) - C\frac{r}{\rho}\vol_{d-1}(S_r), 
\end{eqnarray*}
and since $S_r\subset S$, the result follows.
\end{proof}

\subsection{Exponential inequalities}

\begin{proposition}
\label{prop:vol1}
Fix a sequence $r_n \to 0$.
Let $A\subset M$ be an arbitrary open subset of $M$.
There exists a constant $C$ depending only on $\dM$ such that, for any $\varepsilon>0$, and all $n$ large enough, we have
$$
\pr{\left| \mu_n(A) - \mu(A) \right| \geq \varepsilon} \leq 2 \exp\left(- \frac{n r_n^d \varepsilon^2}{C (1+\varepsilon)}   \right).
$$
In particular, if $n r_n^{d}/\log n \to \infty$, then $\mu_n(A)$ converges almost surely to $\mu(A)$ when $n\to\infty$.
\end{proposition}

\begin{proof}

By the triangle inequality, we have
$$
\left| \mu_n(A) - \mu(A) \right| \leq \left| \mu_n(A) - \expect{\mu_n(A)} \right| + \left| \expect{\mu_n(A)} - \mu(A) \right|.
$$
For all $n$ large enough such that $r_n\leq{\rm reach}(\partial M)$, the second term on the right-hand side (the bias term) is bounded by $C r_n$ with $C$ depending only on $\dM$. 
Indeed, Lemma~\ref{lem:evol} states that the bias is lower than $\mu(A\cap M_{r_n}^c)$. And the tubular neighborhood of $\partial M$ of radius $r_n$, which contains $A\cap M_{r_n}^c$, has a volume bounded by $Cr_n$ by~\eqref{tube}.

Assume that $n$ is large enough such that $2 C r_n \leq \varepsilon$. 
We then apply Theorem~\ref{lem:hoeff}, which is Hoeffding's Inequality for $U$-statistics, to the first term (the deviation term) on the right-hand side with the kernel
$$\phi := \phi_{A,r_n} - \expect{\phi_{A,r_n}(X_1, X_2)}$$
and $t = \omega_d r^d \varepsilon/2$.
The kernel satisfies $\|\phi\|_\infty \leq 1$, and simple calculations yields
$$
{\rm Var}(\phi(X_1, X_2)) \leq \expect{\phi_{A,r_n}(X_1, X_2)^2} \leq \mu(A)\omega_d r_n^d/\tau_M \leq \omega_dr_n^d/\tau_M.
$$
From this we obtain the large deviation bound.
The almost sure convergence is then a simple consequence of the Borel-Cantelli Lemma.
\end{proof}

\begin{proposition} \label{prop:per1}
Fix a sequence $r_n \to 0$.
Let $A$ be an open subset of $M$ with smooth relative boundary and positive reach.
There exists a constant $C$ depending only on $M$ such that, for any $\varepsilon>0$, and for all $n$ large enough,
we have
$$
\pr{\left| \nu_n(A) - \nu(A) \right| \geq \epsilon} \leq 2 \exp\left(- \frac{n r_n^{d+1} \epsilon^2}{C (\nu(A) + \epsilon)} \right).
$$
In particular, if $n r_n^{d+1}/\log n \to \infty$, then
$$\nu_n(A) \to \nu(A), \quad n \to \infty, \quad \text{almost surely}.$$  
\end{proposition}

\begin{proof}
By the triangle inequality, we have
$$
\left| \nu_n(A) - \nu(A) \right| \leq \left| \nu_n(A) - \expect{\nu_n(A)} \right| + \left| \expect{\nu_n(A)} - \nu(A) \right|.
$$
Using the control on the bias in Lemma~\ref{lem:eper}-$(i)$, the second term on the right-hand side goes to 0 as $n\to\infty$.
Then for $n$ large enough, we apply Hoeffding's inequality of Theorem~\ref{lem:hoeff} to the first term on the right-hand side with the kernel
$$\phi := \bar{\phi}_{A,r_n} - \expect{\bar{\phi}_{A,r_n}(X_1, X_2)}$$
and $t := \gamma_{d} r^{d+1} \nu(A) \epsilon/2$.
The kernel satisfies $\|\phi\|_\infty \leq 1$, hence
$$
{\rm Var}(\phi(X_1, X_2)) \leq \expect{\bar{\phi}_{A,r_n}(X_1, X_2)^2} = \expect{\bar{\phi}_{A,r_n}(X_1, X_2)} \leq 2 \gamma_{d} \nu(A) r_n^{d+1}/\tau_M,
$$
where the last inequality follows from upper bound on the bias of Lemma~\ref{lem:eper}-$(i)$ for $n$ large enough.
From this we obtain the large deviation bound, and the almost sure convergence is a consequence of the Borel-Cantelli Lemma.
\end{proof}

\subsection{Proof of Theorem~\ref{theo:pointwise}}
The first statement of Theorem~\ref{theo:pointwise} is an immediate consequence of the exponential inequalities of Propositions~\ref{prop:vol1} and \ref{prop:per1}.

To prove the second statement, under the conditions of Theorem~\ref{theo:pointwise}, for any subset $A$ with smooth relative boundary, with probability one $\lim_n h_n(A) = h(A;M)$ while $h_n(A) \geq \frac{\omega_d}{\gamma_{d} r_n} H(\dG_{n, r_n})$, so that $\limsup_n \frac{\omega_d}{\gamma_{d} r_n} H(\dG_{n, r_n}) \leq h(A;M)$.
Then we obtain the upper bound of Theorem~\ref{theo:pointwise} by taking the infimum over all such subsets $A$.

\section{Proof of Theorems~\ref{theo:estim} and \ref{theo:discretemeasure}: consistent estimation}
\label{sec:unif}

Consistent estimation in the context of Theorem~\ref{theo:estim} is possible because the class $\cR_n$ is sufficiently rich as to include sets that approach Cheeger sets of $M$ and its complexity is controlled, so as to allow for a uniform convergence both in terms of discrete volume and discrete perimeter.  This control on the complexity of $\cR_n$ we exploit in building a covering for $\cR_n$, which is done in Section~\ref{sub:covering}, later used to obtain uniform versions of Propositions~\ref{prop:vol1} and \ref{prop:per1}.
Then Part \textit{(i)} of Theorem~\ref{theo:estim}, which states the convergence of a penalized graph Cheeger constant towards the Cheeger constant of $M$, is proved in Section~\ref{sub:proof_i}. 
Finally, Part \textit{(ii)}, which characterizes the accumulation points of a sequence of minimizing sets, is proved in Section~\ref{sub:proof_ii}. 
The convergence of the discrete measures associated with a sequence of minimizing sets (Theorem~\ref{theo:discretemeasure}) is proved in Section~\ref{sub:proof_3}.

\subsection{Covering numbers}
\label{sub:covering}

For $\rho > 0$, let $\cR_{\rho}$ be the class of open subsets $R \subset (0,1)^d$ with ${\rm reach}(\partial R) \geq \rho$.  Let $d_H(R,R')$ be the Hausdorff distance between two sets $R$ and $R'$, i.e., 
$$d_H(R,R') = \inf\left\{r>0\,:\,R\subset R'\oplus B(r)\quad\text{and}\quad R'\subset R\oplus B(r) \right\}.$$
Denote by $\cN\left(\varepsilon,\cR_{\rho},d_H\right)$ be the covering number of $\cR_{\rho}$ for the Hausdorff distance, i.e., the minimal number of balls of radius $\varepsilon$ for the Hausdorff distance, centered at elements in $\cR_{\rho}$ that are needed to cover $\cR_{\rho}$.

\begin{lemma}
\label{lem:vcan}
(i) There exists a constant $C$ depending only on $d$ such that, for any $\varepsilon>0$ and any $\rho>0$:
$$\log\cN(\varepsilon,\cR_{\rho},d_H) \leq C\left(\frac{1}{\varepsilon}\right)^d.$$
(ii) If $0 < \varepsilon <\rho$, then for any $R$ and $R'$ in $\cR_{\rho}$, if $d_H(R,R')\leq\varepsilon$, then $R\Delta R' \subset \cV(\partial R,\varepsilon) \cap \cV(\partial R',\varepsilon)$.
\end{lemma}

\begin{proof}
Let $x_1,\dots,x_n$ be an $\eps$-packing of $(0,1)^d$, so $\cup_{i=1}^n B(x_i,\varepsilon)$ covers $(0,1)^d$ and $n \leq C \varepsilon^{-d}$ for some constant $C$ depending only on $d$.
For any set $R$ in $\cR_{\rho}$, define
$$I_\varepsilon(R) = \left\{i=1,\dots,n\,:\,B(x_i,\varepsilon)\cap R \neq\emptyset\right\}.$$
Then clearly, by definition of the covering, $R\subset \cup_{i\in I_\varepsilon(R)} B(x_i,\varepsilon)$, and
$$\cup_{i\in I_\varepsilon(R)} B(x_i,\varepsilon) \subset R\oplus B(2\varepsilon).$$
Therefore
$$d_H\left(\cup_{i\in I_\varepsilon(R)} B(x_i,\varepsilon) , R  \right) \leq 2\varepsilon.$$
Since when $R$ ranges in $\cR_{\rho}$, the cardinality of sets of the form $\cup_{i\in I_\varepsilon(R)} B(x_i,\varepsilon)$ is bounded by $2^{n}$, then the collection of Hausdorff balls of radius $2\varepsilon$ and centered set of  the form $\cup_{i\in I} B(x_i,\varepsilon)$, where $I$ is any subset of $\{1, \dots, n\}$, covers  $\cR_{\rho}$.
By doubling the radius of the balls, we can take centers in $\cR_{\rho}$, which proves the first part of the lemma.

The second part follows from the fact that if ${\rm reach}(\partial R) > \rho$, then $\partial R \oplus B(\rho) = \cV(\partial R,\rho)$, assuming, without loss of generality, that $\partial R$ has no boundary.
\end{proof}

We mention that the bound on the $\eps$-entropy of $\cR_\rho$ is rather weak.  Standard results by Kolmogorov and Tikhomirov~\cite{MR0124720} suggest a bound of the form $C (\rho \eps)^{-(d-1)/2}$.  Such a result would change the exponent for $r_n$ in Theorem~\ref{theo:estim} to $(3d+1)/2$.

\subsection{Perimeter bounds of a regular set}

The classical isoperimetric inequality provides a bound of the volume of a Borel set $R$ in terms of its perimeter (see e.g., Evans and Gariepy, 1992):
\begin{equation}
\label{eq:isoperimetric}
d\omega_d^{1/d} \vol_d(R)^{1-1/d}\le \vol_{d-1}(\partial R).
\end{equation}
But, in the case where $\partial R$ has positive reach, the perimeter may in turn be bounded by the volume, as stated in Lemma~\ref{lem:borneperi} below.
The proof uses the following inequality: for every Borel sets $R,S$
\begin{equation}
\vol_{d-1}\big(\partial(R\cup  S)\big)+\vol_{d-1}\big(\partial(R\cap S)\big) \le  \vol_{d-1}(\partial R)+\vol_{d-1}(\partial S).
\label{eq:per1} 
\end{equation}

\begin{lemma}
\label{lem:borneperi}
Let $R$ be a bounded open subset of $\mathbb{R}^d$ with
${\rm reach}(\partial R) = \rho>0$.
Then,
\[\vol_{d-1}(\partial R)\le d \vol_d( R)/\rho.\]
\end{lemma}
\begin{proof}
Since ${\rm reach}(\partial R) = \rho >0$, a ball of radius $\rho$ rolls freely in $R$.
Consequently $R$ can be written as a countable union of balls of radius $\rho$, i.e.,
$$R=\bigcup_{i=1}^\infty B(x_i,\rho).$$
Set $R_n=\cup_{i=1}^n B_i$ where $B_i=B(x_i,\rho)$.

  
Using the decomposition $R_{n+1}=R_n\cup B_{n+1}$, on the one hand we have
$$\vol_d(R_{n+1}) =  \vol_d(R_n\cup B_{n+1})= \vol_d(R_n)+\omega_d \rho^d -  \vol_d(R_n\cap B_{n+1}),$$
and on the other hand, using inequality~\eqref{eq:per1}, we have
$$\vol_{d-1}(\partial R_{n+1}) =
  \vol_{d-1}(\partial(R_n\cup B_{n+1}))
  \le \vol_{d-1}(\partial R_n) + d\omega_d \rho^{d-1} -
  \vol_{d-1}(\partial (R_n \cap B_{n+1})).
$$
Consequently
\begin{eqnarray*}
\vol_{d-1}(\partial R_{n+1}) - \frac d \rho \vol_d(R_{n+1})  & \leq & \vol_{d-1}(\partial R_n) - \frac d \rho \vol_d(R_n)\\
& & +\left[ \frac d \rho \vol_d(R_n\cap B_{n+1}) - \vol_{d-1}(\partial (R_n\cap B_{n+1})) \right].
\end{eqnarray*}
But, using the isoperimetric inequality~\eqref{eq:isoperimetric}, we may write
\begin{align*}
&    \frac{d}{\rho}\vol_d(R_n\cap B_{n+1})-
    \vol_{d-1}\big(\partial (R_n\cap B_{n+1})\big)\\
    &\quad \le
    \frac{d}{\rho}\vol_d(R_n\cap B_{n+1}) - 
    d \omega_d^{1/d}\bigg(\vol_d(R_n\cap B_{n+1})\bigg)^{1-1/d}
    \\
    & \quad\le  \bigg(\vol_d(R_n\cap B_{n+1})\bigg)^{1-1/d} \bigg[ 
    \frac{d}{\rho} \vol_d(R_n\cap B_{n+1})^{1/d}-d\omega_d^{1/d}
    \bigg] \le 0
  \end{align*}
  since, in the last bracket, $\vol_d(R_n\cap B_{n+1})\le 
  \vol_d(B_{n+1})=\omega_d\rho^d$. 
Therefore, for all $n\geq 1$, we have
$$\vol_{d-1}(\partial R_{n+1}) - \frac d \rho \vol_d(R_{n+1})   \leq  \vol_{d-1}(\partial R_n) - \frac d \rho \vol_d(R_n).$$
But since $R_1$ is a ball of radius $\rho$, we have $ \vol_{d-1}(\partial R_1) - d \vol_d(R_1)/\rho  = 0$ and so
\[ \vol_{d-1}(\partial R_n) - \frac d \rho \vol_d(R_n) \leq 0\quad\text{for all $n\geq 1$}.\]
Since $R_n$ converges to $R$ in $L^1$, it follows from the lower semi-continuity of the perimeter, see e.g.~\cite[Prop.~2.3.6]{henrot-pierre}, that $\liminf_n \vol_{d-1}(\partial R_n) \geq \vol_{d-1}(\partial R)$. This concludes the proof.
\end{proof}

\subsection{Exponential inequalities}

We prove the uniform versions of Propositions~\ref{prop:vol1} and~\ref{prop:per1} for the class $\cR_\rho$.  

\begin{proposition}
\label{prop:vol}
There exists a constant $C$ depending only on $\dM$ such that, for any $\eps, r > 0$ and all $n$ satisfying
$nr^d \rho^d \varepsilon^{d+2} > C$ and $\eps > C r$, we have
\begin{equation} \label{eq:vol}
\pr{ \sup_{R\in\cR_{\rho}} \left| \mu_n(R) - \mu(R) \right| \geq\varepsilon} 
\leq 2 \exp\left(- \frac{n r^d \varepsilon^2}{C (1 + \eps)}\right).
\end{equation}
\end{proposition}
\begin{proof}
The bias term is dealt exactly as in Proposition~\ref{prop:vol1}, obtaining 
$$\left|\mathbb{E}\left[\mu_n(R)\right] - \mu(R) \right| \leq C_0 r,$$
valid for all $R\in\cR_{\rho}$, so assuming $\eps > 2 C_0 r$, we may focus on bounding the variance term
$$\mu_n(R) - \mathbb{E}\left[\mu_n(R)\right].$$  
Define the kernel class
\begin{equation}
\label{eq:fn}
\cF = \{\phi_{R,r}\,:\,R\in\cR_{\rho}\},
\end{equation}
where $\phi_{R,r}$ is defined in \eqref{eq:phi}. 
Let $U_n(\phi)$ be the U-process over $\cF$ defined by
$$U_n(\phi) = \frac{1}{n(n-1)}\sum_{i\neq j} \phi(X_i,X_j).$$
Observe that
$$\sup_{R\in\cR_{\rho}}\left| \mu_n(R) - \expect{\mu_n(R)} \right|  = \frac{\tau_M}{\omega_d r^d} \sup_{\phi \in \cF} \left|U_n(\phi) - \mu^{\otimes 2}(\phi)\right|.$$
Consider a minimal covering of $\cR_{\rho}$ of cardinal $K$ by balls centered at elements $R_1,\dots,R_K$ of $\cR_{\rho}$, and of radius $\eta<\rho$ for the Hausdorff distance.
By Lemma \ref{lem:vcan},
$$\log(K) \leq C_1 (1/\eta)^d.$$
For any $R$ in $\cR_{\rho}$, there exists $1\leq k \leq K$ such that $d_{H}(R,R_k) \leq \eta$, which implies that $R\Delta R_k \subset \cV(\partial R_k,\eta)$.
Also, by Lemma \ref{lem:tube}, there exists a constant $C_2$ depending only on the dimension $d$ such that $\vol_d(\cV(\partial R_k,\eta)) \leq C_2\eta/\rho$, for all $1\leq k\leq K$, which implies that
$$\mu\left(\cV(\partial R_k,\eta)\right) \leq C_3 \eta/\rho,\quad\text{for all $1\leq k \leq K$},$$
since $\eta<\rho$, and where $C_3$ now depends on $\dM$.

We have
\begin{eqnarray*}
\left|\phi_{R,r}(x,y) - \phi_{R_k,r}(x,y)\right| & = & \frac{1}{2}\left|\mathbf{1}_R(x) + \mathbf{1}_R(y) - \mathbf{1}_{R_k}(x) - \mathbf{1}_{R_k}(y)\right|\mathbf{1}\left\{\|x-y\|\leq r\right\}\\
 & \leq &  \frac{1}{2}\left(\mathbf{1}_{R\Delta R_k}(x) + \mathbf{1}_{R\Delta R_k}(y)\right) \mathbf{1}\left\{\|x-y\|\leq r\right\}.
\end{eqnarray*}
Next, consider the inequality
$$\left| U_n(\phi_{R,r}) - \mu^{\otimes 2}(\phi_{R,r})\right| \leq \left|U_n(\phi_{R,r})-U_n(\phi_{R_k,r})\right| + \left|U_n(\phi_{R_k,r}) - \mu^{\otimes 2}(\phi_{R_k,r})\right| + \left| \mu^{\otimes 2}(\phi_{R_k,r}) - \mu^{\otimes 2}(\phi_{R,r})\right|.$$

For the double expectations, we have,
\begin{eqnarray*}
\left| \mu^{\otimes 2}(\phi_{R_k,r}) - \mu^{\otimes 2}(\phi_{R,r})\right| & \leq & \mu^{\otimes 2}\left|\phi_{R_k,r} - \phi_{R,r}\right|\\
& = & \mathbb{E}\left[1_{R\Delta R_k}(X_1)\mathbf{1}\left\{\|X_1-X_2\|\leq r\right\}\right]\\
& = & \int_{R\Delta R_k}\mu\left(B(x,r)\right)\mu({\rm d}x)\\
& \leq & \int_{\cV(\partial R_k,\eta)}\mu\left(B(x,r)\right)\mu({\rm d}x)\\
& \leq & \frac{\omega_dr^d}{\tau_M} \mu\left(\cV(\partial R_k,\eta)\right)\\
& \leq & C_4r^d \eta/\rho,
\end{eqnarray*}
with $C_4$ still depending only on $M$.  The last inequality is a consequence of Lemmas~\ref{lem:tube} and~\ref{lem:borneperi}, and the fact that $\vol_d(R_k) \leq 1$ since $R_k \subset (0,1)^d$.

For the empirical averages, we have
\begin{eqnarray*}
\left|U_n(\phi_{R,r})-U_n(\phi_{R_k,r})\right| & \leq & \frac{1}{2} \frac{1}{n(n-1)} \sum_{i\neq j}\left(\mathbf{1}_{R\Delta R_k}(X_i) + \mathbf{1}_{R\Delta R_k}(X_j)\right) \mathbf{1}\left\{\|X_i-X_j\|\leq r\right\}\\
& \leq & \frac{1}{2} \frac{1}{n(n-1)}\sum_{i\neq j}\left(\mathbf{1}_{\cV(\partial R_k,\eta)}(X_i) + \mathbf{1}_{\cV(\partial R_k,\eta)}(X_j)\right) \mathbf{1}\left\{\|X_i-X_j\|\leq r\right\}\\
 & = & U_n\left(\phi_{\cV(\partial R_k,\eta)}\right).
\end{eqnarray*}

Therefore,
$$\sup_{R\in\cR_{\rho}} \left| U_n(\phi_{R,r}) - \mu^{\otimes 2}(\phi_{R,r})\right| \leq \max_{1\leq k \leq K}U_n\left(\phi_{\cV(\partial R_k,\eta)}\right) + C_4\frac{r^d \eta}{\rho} +  \max_{1\leq k \leq K}\left|U_n(\phi_{R_k,r}) - \mu^{\otimes 2}(\phi_{R_k,r})\right|.$$
Consequently, for any $\varepsilon>0$, we may write
\begin{align*}
& \mathbb{P}\left( \sup_{R\in\cR_{\rho}}\left| \mu_n(R) - \expect{\mu_n(R)} \right| \geq \varepsilon \right)\\
 & \quad =\mathbb{P}\left(  \sup_{\phi \in \cF} \left|U_n(\phi) - \mu^{\otimes 2}(\phi)\right| \geq \frac{\omega_dr^d\varepsilon}{\tau_M}\right)\\
 & \quad \leq \mathbb{P}\left(  \max_{1\leq k \leq K}U_n\left(\phi_{\cV(\partial R_k,\eta)}\right) \geq \frac{\omega_dr^d\varepsilon}{2\tau_M} - C_4\frac{r^d \eta}{\rho}\right)
 + \mathbb{P}\left(\max_{1\leq k \leq K}\left|U_n(\phi_{R_k,r}) - \mu^{\otimes 2}(\phi_{R_k,r})\right| \geq \frac{\omega_dr^d\varepsilon}{2\tau_M} \right)\\
 & \quad \leq K \max_{1\leq k \leq K} \mathbb{P}\left(U_n\left(\phi_{\cV(\partial R_k,\eta)}\right) \geq \frac{\omega_dr^d\varepsilon }{2\tau_M}- C_4\frac{r^d \eta}{\rho}\right)
  + K\max_{1\leq k \leq K} \mathbb{P}\left(\left|U_n(\phi_{R_k,r}) - \mu^{\otimes 2}(\phi_{R_k,r})\right| \geq \frac{\omega_dr^d\varepsilon}{2\tau_M} \right),
\end{align*}
by the union bound.
To bound the first term, note first that
$$\operatorname{Var}\left(\phi_{\cV(\partial R_k,\eta)}(X_1,X_2)\right) \leq \mathbb{E}\left[ \phi_{\cV(\partial R_k,\eta)}(X_1,X_2)^2\right] \leq \mathbb{E}\left[ \phi_{\cV(\partial R_k,\eta)}(X_1,X_2)\right],$$
with
$$\mathbb{E}\left[ \phi_{\cV(\partial R_k,\eta)}(X_1,X_2)\right]   \leq  \frac{\omega_dr^d}{\tau_M} \mu\left(\cV(\partial R_k,\eta)\right) \leq C_4 \frac{r^d\eta}{\rho},$$
for the same reasons as above.
Now take $\eta = \rho \min(\omega_d \varepsilon/(8C_4\tau_M), 1)$.
Then, for any $1\leq k \leq K$, by Hoedffding's inequality for U-statistics (Theorem~\ref{lem:hoeff}), we have,
\begin{eqnarray*}
\mathbb{P}\left(U_n\left(\phi_{\cV(\partial R_k,\eta)}\right) \geq \frac{\omega_dr^d\varepsilon}{2\tau_M} - C_4\frac{r^d\eta}{\rho}\right) & \leq &\mathbb{P}\left( U_n\left(\phi_{\cV(\partial R_k,\eta)}\right) - \mathbb{E}\left[U_n\left(\phi_{\cV(\partial R_k,\eta)}\right)\right] \geq \frac{\omega_dr^d\varepsilon}{4\tau_M}\right)\\
& \leq & \exp\left(-\frac{n(\omega_dr^d\varepsilon/4\tau_M)^2}{5(C_4 r^d \eta/\rho) + 3(\omega_dr^d\varepsilon/4\tau_M)}\right)\\
& \leq & \exp\left(-\frac{n r^d\varepsilon}{C_5}\right),
\end{eqnarray*}
for a constant $C_5 > 0$ depending only on $M$.
To bound the second term, since
$$\operatorname{Var}\left(\phi_{R_k,r}(X_1,X_2)\right) \leq \mathbb{E}\left[\phi_{R_k,r}(X_1,X_2)\right] \leq \omega_dr^d/\tau_M,$$
we may apply Lemma \ref{lem:hoeff} again to obtain the bound
\begin{eqnarray*}
\mathbb{P}\left(\left|U_n(\phi_{R_k,r}) - \mu^{\otimes 2}(\phi_{R_k,r})\right| \geq \frac{\omega_dr^d\varepsilon}{2\tau_M} \right) & \leq & \exp\left(-\frac{n(\omega_dr^d\varepsilon/2\tau_M)^2}{5\omega_dr^d + 3(\omega_dr^d\varepsilon/2\tau_M)} \right)\\
& \leq & \exp\left(-\frac{n r^d\varepsilon^2}{C_6 (1 + \eps)}\right),
\end{eqnarray*}
for a constant $C_6 > 0$ depending only on $M$.

With the choice of $\eta$ as above, the cardinal $K$ of the covering is such that $\log(K) \leq C_7(\varepsilon\rho)^{-d}$, for some constant $C_7$ depending only on $\dM$, and we obtain the bound
\begin{align*}
& \mathbb{P}\left( \sup_{R\in\cR_{\rho}}\left| \mu_n(R) - \expect{\mu_n(R)} \right| \geq \varepsilon \right)\\
& \quad \leq K\exp\left(-\frac{n r^d\varepsilon}{C_5} \right) + K\exp\left(-\frac{n r^d\varepsilon^2}{C_6 (1 + \eps)} \right)\\
&\quad\leq 2\exp\left( C_7 (\varepsilon\rho)^{-d} - \frac{n r^d \varepsilon^2}{C_8 (1+\eps)}  \right)\\
&\quad\leq 2\exp\left(-\frac{nr^d\varepsilon^2}{C_9 (1+\eps)}\right),
\end{align*}
if $nr^d\varepsilon^{d+2} \rho^d > C_9$, for a constant $C_9$ depending only on $M$.
\end{proof}

For the perimeter, we only control the variance, as the bias may not be controlled uniformly over $\mathcal{R}_\rho$.  Indeed, consider the case where $M$ is a hypercube with rounded corners so as to satisfy the condition on its reach, and let $R$ be another hypercube with rounded corners included in $M$ sharing one of its faces with $M$.  Then given a sample $X_1, \dots, X_n$, it is possible to translate $R$ inside $M$ just enough that the translate does not share a boundary with $M$, while its discrete volume and perimeter are left equal to those of $R$.

\begin{proposition}
\label{prop:perim}
There exists a constant $C$ depending only on $\dM$ such that, for any $\eps > 0$, $\rho < 1$, $r < \min({\rm reach}(M), \rho/2)$ and all $n$ satisfying
$n r^{2d+1} \rho^{d+1} \varepsilon^{d+2} > C,$
we have
$$\pr{ \sup_{R\in\cR_{\rho}} \left| \nu_n(R) - \mathbb{E}[\nu_n(R)] \right| \geq \varepsilon} 
\leq 2\exp\left( -\frac{nr^{d+1}\rho\varepsilon^2}{C(1+\rho\eps)}  \right).$$
\end{proposition}

\begin{proof}
The proof follows that of Proposition~\ref{prop:vol}, with the symmetric kernel $\bar{\phi}_{R,r}$ defined in \eqref{eq:barphi} and the class $\bar{\cF}$ defined in \eqref{eq:fn} with $\phi_{R,r}$ replaced by $\bar{\phi}_{R,r}$.
Observe that
$$
\left| \nu_n(R) - \expect{\nu_n(R)} \right|  = \frac{\tau_M}{\gamma_{d} r^{d+1}} \sup_{\phi \in \bar{\cF}} \left|U_n(\phi) - \mu^{\otimes 2}(\phi)\right|.
$$
As in the proof of Proposition ~\ref{prop:vol}, we start with a minimal covering of $\cR_{\rho}$ of cardinal $K$ by balls of radius $\eta$ for the Hausdorff distance.
For any $R$ in $\cR_{\rho}$ at a Hausdorff distance no more than $\eta$ of an element $R_k$ of the covering, we have
\begin{eqnarray*}
\left|\mathbf{1}_R(x)\mathbf{1}_{R^c}(y) - \mathbf{1}_{R_k}(x)\mathbf{1}_{R_k^c}(y) \right| & \leq & \left|\mathbf{1}_R(x) - \mathbf{1}_{R_k}(x)\right|\mathbf{1}_{R^c}(y) + \mathbf{1}_{R_k}(x)\left| \mathbf{1}_{R^c}(y) - \mathbf{1}_{R_k^c}(y)\right|\\
 & = & \mathbf{1}_{R\Delta R_k}(x)\mathbf{1}_{R^c}(y) + \mathbf{1}_{R\Delta R_k}(y)\mathbf{1}_{R_k}(x)\\
 & \leq & \mathbf{1}_{R\Delta R_k}(x) + \mathbf{1}_{R\Delta R_k}(y).
\end{eqnarray*}
Hence,
$$\left| \bar{\phi}_{R,r}(x,y) - \bar{\phi}_{R_k,r}(x,y)\right|  \leq 2 \phi_{R\Delta R_k,r}(x,y) \leq 2 \phi_{\cV(\partial R_k,\eta)}(x,y), $$
and therefore, following the same arguments,
$$
\left|\mu^{\otimes 2} (\bar{\phi}_{R,r}) -\mu^{\otimes 2}( \bar{\phi}_{R_k,r})\right| \leq 2 \mu^{\otimes 2} (\phi_{\cV(\partial R_k,\eta)}) \leq C_1 r^d \eta/\rho,
$$
for a constant $C_1$ depending only on $M$; and also,
$$\left|U_n\left(\bar{\phi}_{R,r}\right) - U_n\left(\bar{\phi}_{R_k,r}\right)\right|  \leq 2 U_n\left(\phi_{\cV(\partial R_k,\eta)}\right).$$
Hence
\begin{align*}
& \mathbb{P}\left(\sup_{R\in\cR_{\rho}}\left|\nu_n(R) - \mathbb{E}\left[\nu_n(R)\right] \right| \geq \varepsilon \right)\\
& \quad \leq K\max_{1\geq k \geq K} \mathbb{P}\left(U_n\left(\phi_{\cV(\partial R_k,\eta)}\right) \geq \frac{\gamma_dr^{d+1}\varepsilon}{4\tau_M} - C_1\frac{r^d \eta}{2\rho} \right)\\
&\quad\quad  + K\max_{1\geq k \geq K} \mathbb{P}\left( \left| U_n(\bar{\phi}_{R_k,r}) - \mu^{\otimes2}(\bar{\phi}_{R_k,r}) \right|   \geq \frac{\gamma_dr^{d+1}\varepsilon}{2\tau_M} \right).\\
\end{align*}
Take $\eta = \rho \min\big(\gamma_d r \eps/(4C_1\tau_M), 1\big)$.
For the first term, for any $1 \leq k \leq K$, we have,
\begin{eqnarray*}
\mathbb{P}\left(U_n\left(\phi_{\cV(\partial R_k,\eta)}\right) \geq \frac{\gamma_dr^{d+1}\varepsilon}{4\tau_M} - C_1 \frac{r^d \eta}{2\rho} \right) & \leq & \mathbb{P}\left(U_n\left(\phi_{\cV(\partial R_k,\eta)}\right) - \mathbb{E}\left[U_n\left(\phi_{\cV(\partial R_k,\eta)}\right)\right] \geq \frac{\gamma_dr^{d+1}\varepsilon}{8\tau_M} \right)\\
& \leq & \exp\left(-\frac{n (\gamma_dr^{d+1}\varepsilon/8\tau_M)^2}{5(C_1 r^d \eta/\rho) + 3(\gamma_dr^{d+1}\varepsilon/8\tau_M)}\right)\\
& = & \exp\left(-\frac{n r^{d+1}\varepsilon^2}{C_2 (1 + \eps)}\right),
\end{eqnarray*}
for some constant $C_2 > 0$ depending only on $M$.
For the second term, since by Lemma \ref{lem:eper}, when $r\leq\rho/2$,
$$\operatorname{Var}\left(\bar{\phi}_{R_k,r}\right) \leq C_3 r^{d+1} / \rho,$$
for a constant $C_3$ depending only on $\dM$, we have
\begin{eqnarray*}
\mathbb{P}\left( \left| U_n(\bar{\phi}_{R_k,r}) - \mu^{\otimes2}(\bar{\phi}_{R_k,r}) \right|   \geq \frac{\gamma_dr^{d+1}\varepsilon}{2\tau_M} \right) & \leq & \exp\left(-n \frac{(\gamma_dr^{d+1}\varepsilon/2\tau_M)^2}{5(C_3 r^{d+1}/\rho) + 3(\gamma_dr^{d+1}\varepsilon/2\tau_M)} \right)\\
 & \leq & \exp\left(-\frac{n r^{d+1} \rho \varepsilon^2 }{C_4 (1 + \rho \eps)} \right),
\end{eqnarray*}
for a constant $C_4>0$ depending only on $M$.
Finally, with the choice of $\eta$ as above, the cardinal $K$ of the covering is such that $\log(K) \leq C_5 (r \rho \varepsilon)^{-d}$, for $C_5$ depending only on $M$.

Then
$$K\max_{1\geq k \geq K} \mathbb{P}\left(U_n\left(\phi_{\cV(\partial R_k,\eta)}\right) \geq \frac{\gamma_dr^{d+1}\varepsilon}{4\tau_M} - C_1 \frac{r^d \eta}{\rho}) \right) \leq \exp\left( -\frac{nr^{d+1}\varepsilon^2}{C_6 (1+\eps)}  \right),$$
if $n r^{2d+1} \rho^d \varepsilon^{d+2} > C_6,$
and
$$K\max_{1\geq k \geq K} \mathbb{P}\left( \left| U_n(\bar{\phi}_{R_k,r}) - \mu^{\otimes2}(\bar{\phi}_{R_k,r}) \right|   \geq \frac{\gamma_dr^{d+1}\varepsilon}{2\tau_M} \right) \leq \exp\left(-\frac{n r^{d+1} \rho \varepsilon^2 }{C_7 (1 + \rho \eps)} \right),$$
if $nr^{2d+1}\rho^{d+1} \varepsilon^{d+2} > C_7,$
where $C_6$ and $C_7$ depend on $M$ only.
Combining these inequalities, we conclude.
\end{proof}

\subsection{A uniform control on $h_n(A)$}
\label{sub:uniformcontrol}

As we argued earlier, the boundary of $M$ makes a uniform convergence of the perimeters of sets in $\cR_n$ impossible.  Our way around that is to compare the discrete perimeter of a set $R$ with its perimeter inside $M_{r_n}$, thus avoiding the boundary of $M$, i.e., $\vol_{d-1}(\partial R \cap M_{r_n})$, leading to a comparison between $h_n(R)$ and $h(R;M_{r_n})$.  We relate the latter to $h(R; M)$ in Section~\ref{sub:cont_Cheeger}.   
 
\begin{lemma}
\label{lem:approxh}
Under the conditions of Theorem~\ref{theo:estim}, with probability one, we have:
\begin{equation}
\label{eq:liminfdelinf}
\liminf_{n\to\infty} \inf_{R\in\cR_n}\left(h_n(R) - h(R; M_{r_n})\right)\ge 0.
\end{equation}
\end{lemma}
\begin{proof} 
Take $R\in\cR_n$ and define 
$$\lambda_n(R) = \min(\mu_n(R), \mu_n(R^c)), \quad \lambda_n^*(R) = \frac1{\tau_M} \min(\vol_d(R \cap M_{r_n}), \vol_d(R^c \cap M_{r_n})),$$
as well as 
$$\nu_n^*(R) = \frac1{\tau_M} \vol_{d-1}(\partial R \cap M_{r_n}).$$  
Then
\begin{eqnarray*}
h_n(R) - h(R; M_{r_n}) 
&=& \frac{1}{\lambda_n(R)} (\nu_n(R) - \nu_n^*(R)) + \frac{\nu_n^*(R)}{\lambda_n(R)\lambda_n^*(R)} (\lambda_n^*(R) - \lambda_n(R)) \\
& =: & \zeta_n(R) + \xi_n(R).
\end{eqnarray*}
Define the event
\[
\Omega_n = \left\{\frac{1}{2} \leq \frac{\lambda_n(R)}{\lambda_n^*(R)} \leq
\frac{3}{2}, \forall R \in \cR_n\right\}.
\]
We will see that $\pr{\Omega_n} \to 1$.

\paragraph{Bounding $\zeta_n(R)$.} By definition of $\cR_n$, the sets $R$ and $R^c$ contain each a ball of radius $\rho_n$, and by Lemma~\ref{lem:ball-vol}, the volume of the intersection of this ball with $M_{r_n}$ is bounded from below by $C_1 \rho_n^d$, for a constant $C_1$ depending only on $M$.
Hence, 
\begin{equation}
\label{eq:muR_n}
\lambda_n^*(R) \geq C_1\rho_n^d.
\end{equation}
Also, on $\Omega_n$, $\lambda_n(R) \geq \lambda_n^*(R)/2$.  These last two inequalities being valid for all $R \in \cR_n$, for $\eps > 0$ we have 
\begin{eqnarray*}
I_1 := \pr{\left[\inf_{R \in \cR_n} \zeta_n(R) < -\eps\right] \cap \Omega_n}
& \leq & \mathbb{P}\left[ \inf_{R\in\cR_n}\left(\nu_n(R)-\nu_n^*(R)\right) < -C_2 \eps\rho_n^d\right]\\
& \leq & \mathbb{P}\left[ \inf_{R\in\cR_n}\big(\nu_n(R)-\mathbb{E}[\nu_n(R)]\big) + \inf_{R\in\cR_n}\big( \mathbb{E}[\nu_n(R)] - \nu_n^*(R)\big) < -C_2 \eps\rho_n^d\right],
\end{eqnarray*}
for a constant $C_2=C_1/2 > 0$.
Using the bias bounds of Lemma~\ref{lem:eper} together with the perimeter bound in Lemma~\ref{lem:borneperi}(ii), we have
$$ \inf_{R\in\cR_n}\big( \mathbb{E}[\nu_n(R)] - \nu_n^*(R)\big) \geq -C_3 \frac{r_n}{\rho_n^2}.$$
Hence, since $r_n = o(\rho_n^\alpha)$ for any $\alpha > 0$, for $\eps$ fixed and $n$ large enough, we have
by assumption, for all $n$ large enough,
\begin{equation*}
I_1 \leq \mathbb{P}\left[ \inf_{R\in\cR_n}\left(\nu_n(R)-\mathbb{E}[\nu_n(R)]\right)   < -C_2\eps\rho_n^d/2 \right]
\leq  \mathbb{P}\left[\sup_{R\in\cR_{\rho_n}}\left|\nu_n(R)-\mathbb{E}[\nu_n(R)]\right| > C_2\eps\rho_n^d/2 \right],
\end{equation*}
where the second inequality comes from the fact that $\cR_n \subset \cR_{\rho_n}$.  By the fact that $n r_n^{2d+1} \rho_n^\alpha \to \infty$ for any $\alpha > 0$, the conditions of Proposition~\ref{prop:perim} are satisfied, so that 
$$I_1 \leq C_4\exp\left(-\frac{nr_n^{d+1}\rho_n^{2d+1}\varepsilon^2}{C_4 (1+\eps)}\right),$$
for some constant $C_4>0$ and all $n$ large enough.
At last, we have
$$\frac{nr_n^{d+1}\rho_n^{2d+1}}{\log(n)} = n r_n^{2d+1} \frac{\rho_n^{2d+1} r_n^{-d}}{\log(n)} \to +\infty,$$
since $r_n = o(\rho_n^\alpha)$ for any $\alpha > 0$ and $r_n \to 0$ polynomially in $n$, we deduce that, for all $\eps>0$,
\begin{equation}
\label{eq:I1}
\sum_n \mathbb{P}\left[ \left[\inf_{R\in\cR_n} \zeta_n(R) < - \eps\right] \cap \Omega_n \right] < \infty.
\end{equation}

\paragraph{Bounding $\xi_n(R)$.}  (We reset the constants, except for $C_1$.)  By the perimeter bound of Lemma~\ref{lem:borneperi}, we have 
$$\nu_n^*(R) \leq \frac{\vol_{d-1}(\partial R)}{\tau_M} \leq d \frac{\vol_{d}(R)}{\tau_M \rho_n} = C_2/\rho_n,$$
for a constant $C_2  > 0$ depending only on $M$.
So, together with \eqref{eq:muR_n} and the fact that, on $\Omega_n$, $\lambda_n(R) \geq \lambda_n^*(R)/2$, 
$$\frac{\nu_n^*(R)}{\lambda_n(R)\lambda_n^*(R)} \leq C_3 \rho_n^{-2d-1},$$ 
for all $R$ in $\cR_n$.  It follows that
\begin{equation}
\label{eq:I2-0}
I_2 := \pr{\left[\inf_{R \in \cR_n} \xi_n(R) < -\eps\right] \cap \Omega_n} \leq \mathbb{P}\left( \sup_{R\in\cR_n} \left| \lambda_n(R)-\lambda_n^*(R)\right| > \frac{\rho_n^{2d+1}\eps}{C_3}   \right).
\end{equation}
Define
$$\mu_n^*(R) = \frac{\vol_d(R \cap M_{r_n})}{\tau_M}.$$
Then 
\begin{eqnarray*}
\left| \lambda_n(R)-\lambda_n^*(R)\right| 
&\leq& \left| \mu_n(R)-\mu_n^*(R)\right| + \left| \mu_n(R^c)-\mu_n^*(R^c)\right| \\
&\leq& \left| \mu_n(R)-\mu(R)\right| + \left| \mu_n(R^c)-\mu(R^c)\right| + 2 \mu(M_{r_n}^c),
\end{eqnarray*}
with $\mu(M_{r_n}^c) \leq C_4 r_n$ by \eqref{tube}.
For $\eps > 0$ fixed and $n$ large enough, $2 C_4 r_n \leq \rho_n^{2d+1} \eps/C_3$, again by the fact that $\rho_n \to 0$ sub-polynomially in $r_n$.  
We therefore obtain that
\begin{equation*}
I_2 \leq 2 \mathbb{P}\left(\sup_{R\in\cR_{\rho_n}}\left|\mu_n(R) - \mu(R)\right| > \frac{\rho_n^{2d+1}\eps}{4 C_3}   \right),
\end{equation*}
where we used the fact that $R^c \in \cR_n$ when $R \in \cR_n$, together with $\cR_n \subset \cR_{\rho_n}$.  We then apply Proposition~\ref{prop:vol}, whose conditions are satisfied for $\eps > 0$ fixed and $n$ large enough, again because $\rho_n \to 0$ very slowly, arriving at
$$I_2 \leq C_4\exp\left(-\frac{n r_n^d \rho_n^{4d+2} \eps^2}{C_4 (1 +\eps)}\right),$$
for some constant $C_4>0$ and all $n$ large enough.
As before, when $\eps$ is fixed, the exponent is a positive power of $n$, so that 
\begin{equation}
\label{eq:I2}
\sum_n \mathbb{P}\left[ \left[\inf_{R\in\cR_n}  \xi_n(R) < - \eps\right] \cap \Omega_n \right] < \infty.
\end{equation}

\paragraph{Bounding $\pr{\Omega_n^c}$.} Since $\lambda_n^*(R)>C\rho_n^d$ for some $C$ uniformly over $R\in\cR_n$ (see \eqref{eq:muR_n} above), we have
\begin{align*}
\mathbb{P}\left(\Omega_n^c\right) & =  \mathbb{P}\left( \sup_{R\in\cR_n} \frac{|\lambda_n(R) - \lambda_n^*(R)|}{\lambda_n^*(R)} > \frac{1}{2}\right)\\
& \leq  \mathbb{P}\left( \sup_{R\in\cR_n} \left|\lambda_n(R) - \lambda_n^*(R)\right| > C\rho_n^d\right).
\end{align*}
We then proceed as in bounding \eqref{eq:I2-0}, obtaining 
\begin{equation}
\label{eq:I3}
\sum_n\mathbb{P}\left(\Omega_n^c\right) < \infty.
\end{equation}

\paragraph{Conclusion.}
We have 
\[ 
\mathbb{P}\left[\inf_{R\in\cR_n} \left( h_n(R) -  h(R; M_{r_n}) \right) < -2\varepsilon\right]
\leq
\mathbb{P}\left[ \left[\inf_{R\in\cR_n} \zeta_n(R) < - \eps\right] \cap \Omega_n \right]
+ \mathbb{P}\left[ \left[\inf_{R\in\cR_n}  \xi_n(R) < - \eps\right] \cap \Omega_n \right] 
+ \pr{\Omega_n^c},
\]
so that the left-hand side is summable.  Therefore, we conclude by applying the Borel-Cantelli lemma.
\end{proof}

\subsection{Some continuity of the Cheeger constant}
\label{sub:cont_Cheeger}

Our proof of Theorem~\ref{theo:estim} relies on continuity properties of the normalized cut and of the Cheeger constant. 
Lemma~\ref{lem:lip-cheeger} below compares the conductance function on $M$ and on a bi-Lipschitz deformation of $M$. 
For a Lipschitz map $f$, let $\|f\|_{\rm Lip}$ denote its Lipschitz constant.
If $f$ is bi-Lipschitz, we define its condition number by ${\rm cond}(f) := \|f\|_{\rm Lip} \, \|f^{-1}\|_{\rm Lip}$.
Lemma~\ref{lem:near-iso} below states that $M_r$ is a bi-Lipschitz deformation of $M$, hence Lemma~\ref{lem:lip-cheeger} yields the continuity property of Proposition~\ref{prop:cheeger-cont}.

\begin{lemma} \label{lem:lip-cheeger}
Let $f$ be a bi-Lipschitz on $\dM$.  Then for any $A \subset \dM$ measurable,
$$
\max\left\{\frac{h(f(A); f(\dM))}{h(A; \dM)}, \frac{h(A; \dM)}{h(f(A); f(\dM))}\right\} \leq {\rm cond}(f)^d.
$$
\end{lemma}

\begin{proof}
For any $A \subset \dM$, $\partial f(A) = f(\partial A)$ and $f(A)^c \cap f(\dM) = f(A^c \cap \dM)$, and if $A$ is measurable, for $k = 1, \dots, d$, 
\begin{equation} \notag 
\|f^{-1}\|_{\rm Lip}^{-k} \, \vol_k(A) \leq \vol_k(f(A)) \leq \|f\|_{\rm Lip}^{k} \, \vol_k(A).
\end{equation}
Therefore, 
\begin{eqnarray*}
h(f(A); f(\dM)) 
&=& \frac{\vol_{d-1}(f(\partial A \cap \dM))}{\min\{\vol_d(f(A)), \vol_d(f(A^c \cap \dM))\}} \\
&\leq& \frac{\|f\|_{\rm Lip}^{d-1} \vol_{d-1}(\partial A \cap \dM)}{\|f^{-1}\|_{\rm Lip}^{-d} \min\{\vol_d(A), \vol_d(A^c \cap \dM)\}} \\
&\leq& {\rm cond}(f)^d \, h(A; \dM).
\end{eqnarray*}
And vice-versa.
\end{proof}

\begin{lemma} \label{lem:near-iso}
Fix $r<s\le \operatorname{reach}(\partial M)$.
Then there is a bi-Lipschitz map between $\dM_r$ and $\dM$ that leaves $\dM_s$ unchanged, and with condition number at most $(1 + 2r/(s-r))^2$.  
\end{lemma}

\begin{proof}
For $x$ in $\dM$ such that $\delta(x) := {\rm dist}(x, \partial\dM) < s$, let $\xi(x) \in \dM$ be its metric projection onto $\partial \dM$ and $u_x$ be the unit normal vector of $\dM$ at $\xi(x)$ pointing outwards.  
We define the map 
$$
f_r: \dM_r \mapsto \dM, \quad f_r(x) = x + \frac{r (s -\delta(x))_+}{s-r} \, u_x,
$$
where $a_+$ denotes the positive part of $a \in \bbR$.  By construction, $f$ is one-to-one, with inverse
$$
f_r^{-1}: \dM \mapsto \dM_r, \quad f_r^{-1}(x) = x - \frac{r (s -\delta(x))_+}{s} \, u_x.
$$
By~\cite[Thm.~4.8(1)]{MR0110078}, $\delta$ is Lipschitz with constant at most 1, therefore so is $x \mapsto (s -\delta(x))_+$; and since the reach bounds the radius of curvature from below~\cite[Thm.~4.18]{MR0110078}, $x \mapsto u_x$ is Lipschitz with constant at most $1/\operatorname{reach}(\partial M)$.  Therefore, using the fact that $(s -\delta(x))_+ \leq s$ and $\|u_x\| = 1$, $f_r$ and $f_r^{-1}$ are Lipschitz with constants at most $1 + 2r/(s-r)$ and $1 + 2r/s$ respectively. 
\end{proof}

\begin{proposition} \label{prop:cheeger-cont}
We have
$$
H(\dM_r) = (1 + O(r)) \, H(\dM), \quad r \to 0.
$$
\end{proposition}
\begin{proof}
From Lemmas~\ref{lem:lip-cheeger} and~\ref{lem:near-iso}, we deduce that
$$
\max\left\{\frac{H(\dM_r)}{H(\dM)}, \frac{H(\dM)}{H(\dM_r)}\right\} \leq (1 + 2r/(\rho_\dM-r))^{2d},
$$
for any $r < \rho_\dM := {\rm reach}(\partial \dM)$, which immediately yields the desired result. 
\end{proof}


\subsection{$L^1$-metric on Borel sets}
\label{sub:l1}

We will use the $L^1$-metric on Borel subsets of $\bbR^d$, defined by $\vol_d(A \Delta B) = \int \left|{\bf 1}_A(x) -{\bf 1}_B(x)\right| \, {\rm d}x$. 
This metric comes from the bijection between Borel sets $A$ and their indicator functions $\mathbf 1_A$, endowed with the $L^1$-topology.
Strictly speaking, this is a semi-metric on Borel subsets of $\bbR^d$ since
$\vol_d(A\Delta B) = 0$ if and only if $A\Delta B$ is a null set. 

The following propositions are adapted from~\cite[Thm.~2.3.10]{henrot-pierre} and~\cite[Prop.~2.3.6]{henrot-pierre} respectively. 
Proposition~\ref{prop:mp2} is a compactness criterion, and Proposition~\ref{prop:mp1} results from lower semi-continuity of the perimeter measure with respect to $L^1$-metric.

\begin{proposition}
\label{prop:mp2}
Let $(E_n)$ be a sequence of measurable subsets of $\dM$.
Suppose that
$$\limsup_{n \to \infty} \vol_{d-1}(\partial E_n \cap \dM) < \infty.$$
Then $(E_n)$ admits a subsequence converging for the $L^1$-metric.
\end{proposition}

\begin{proposition} \label{prop:mp1}
Let $E_n$ and $E$ be bounded measurable subsets of $\dM$ such that ${E_n} \to E$ in $L^1$, and $h(E;M) < \infty$.  Then
$$\liminf_n h(E_n;M) \geq h(E;M).$$
\end{proposition}

\subsection{Proof of $(i)$ in Theorem~\ref{theo:estim}}
\label{sub:proof_i}

\paragraph{Lower bound.}
For each $n$, let $R_n \in \cR_n$ be such that
$$h_n^\ddag(R_n) = \min_{R\in\cR_n}h_n^\ddag(R).$$
Then
\begin{eqnarray*}
h_n^\ddag(R_n) - H(M) & = & \left[h_n^\ddag(R_n) - h(R_n; M_{r_n})\right] + \left[h(R_n; M_{r_n}) -H(M_{r_n})\right] + \left[H(M_{r_n}) -H(M)\right] \\
 & \geq & \inf_{R \in\cR_n} \left( h_n(R) -  h(R; M_{r_n}) \right) +  \left[H(M_{r_n}) -H(M)\right],
\end{eqnarray*}
since $\left[h(R_n; M_{r_n}) -H(M_{r_n})\right] \geq 0$ by definition of $H(M_{r_n})$.
On the last line, by Lemma~\ref{lem:approxh}, the first term has a non-negative inferior limit, and by Proposition~\ref{prop:cheeger-cont}, the second term tends to zero.  Hence,
\begin{equation}
\label{eq:liminfdehnddag}
\liminf_{n\to\infty}\min_{R\in\cR_n}h_n^\ddag(R) \geq H(M)\quad\text{a.s.}
\end{equation}

\paragraph{Upper bound.}
To obtain the matching upper bound, fix a subset $A\subset M$ with smooth relative boundary and such that $0<\vol_d(A)\leq \vol_d(M\backslash A)<\vol_d(M)$.
Then, for $n$ large enough, there exists $R_n$ in $\cR_n$ such that $R_n\cap M = A$, implying that
$$\min_{R\in\cR_n}h_n^\ddag(R) \leq h_n(A).$$
By Theorem~\ref{theo:pointwise}, $h_n(A) \to h(A;M)$ almost surely, so that
$$\limsup_{n\to\infty} \min_{R\in\cR_n}h_n^\ddag(R) \leq h(A;M)\quad\text{a.s.}$$
By minimizing over $A$, we obtain
\begin{equation}
\label{eq:limsupdehnddag}
\limsup_{n\to\infty}\min_{R\in\cR_n}h_n^\ddag(R) \leq H(M)\quad\text{a.s.}
\end{equation}

Combining the lower and upper bounds, \eqref{eq:liminfdehnddag} and \eqref{eq:limsupdehnddag}, we conclude that
\begin{equation}
\label{eq:ETVOILA}
\lim_{n\to\infty} \min_{R\in\cR_n}h_n^\ddag(R) = H(M)\quad\text{a.s.}
\end{equation}

\subsection{Proof of $(ii)$ in Theorem~\ref{theo:estim}}
\label{sub:proof_ii}
Let $R_n$ be a sequence in $\cR_n$ satisfying 
$$h_n^\ddag(R_n) = \min_{R\in\cR_n}h_n^\ddag(R),$$
and set $A_n = R_n \cap M$.
Fix a subset $A^0 \subset \dM$ with smooth relative boundary and such that $h(A^0) < \infty$.
Then for $n$ large enough, there exists $R$ in $\cR_n$ such that $A^0 = R\cap M$.
Hence $h_n(A_n) \leq h_n(A^0)$ and since $h_n(A^0) \to h(A^0)$ by Theorem~\ref{theo:pointwise}, we have
$$\limsup_{n \to \infty} \vol_{d-1}(A_n) \leq \limsup_{n \to \infty} h(A_n) \min\{\vol_d(A_n), \vol_d(A_n^c \cap \dM)\} \leq h(A^0) \vol_d(\dM)/2.$$ 
Therefore by compactness of the class of sets with bounded perimeters (Proposition~\ref{prop:mp2}), with probability one, $\{A_n\}$ admits a subsequence converging in the $L^1$-metric.\\

On the one hand,
$$
h(A_n; M_{r_n}) -H(M) = \left[h(A_n; M_{r_n}) - H(M_{r_n})\right] + \left[H(M_{r_n}) -H(M)\right],
$$
where the first difference term on the right-hand side is non-negative by definition, while the second difference term tends to zero by Proposition~\ref{prop:cheeger-cont}. 
So that with probability one:
$$\liminf_{n \to \infty} h(A_n; \dM_{r_n})  \geq H(M).$$
On the other hand,
\begin{eqnarray*}
h(A_n;M_{r_n}) - H(M) & = & \left[h(A_n; M_{r_n}) -h_n^\ddag(A_n)\right] + \left[h_n^\ddag(A_n) - H(M)\right]\\
 & \leq & - \inf_{R\in\cR_n} \left(h_n^\ddag(R) - h(R;M_{r_n})\right) + \left[h_n^\ddag(A_n) - H(M)\right]
\end{eqnarray*}
so that
$$\limsup_{n\to\infty} h(A_n;M_{r_n}) - H(M) \leq - \liminf_{n\to\infty}\inf_{R\in\cR_n} \left(h_n^\ddag(R) - h(R;M_{r_n})\right) + \left[h_n^\ddag(A_n) - H(M)\right]$$
which goes to 0 as $n\to\infty$ from \eqref{eq:liminfdelinf} and \eqref{eq:ETVOILA}.
Hence
$$\lim_{n\to\infty} h(A_n;M_{r_n}) \to H(M)\quad\text{a.s.}$$

Now let $f_{n}$ denote the bi-Lipschitz function mapping $\dM_{r_n}$ to $\dM$ defined in Lemma~\ref{lem:near-iso} with $r$ and $s$ replaced by $r_n$ and $s_n$, where $s_n/r_n \to \infty$.
Define $B_n = f_n(A_n\cap M_{r_n})$.
By Lemmas~\ref{lem:lip-cheeger} and~\ref{lem:near-iso}, we have
$$h(B_n;M) \leq \left(1 + \frac{2r_n}{s_n-r_n}\right)^{2d} h(A_n ; M_{r_n}),$$
so that $h(B_n;M) \to H(M)$ almost surely as $n\to\infty$.
Moreover, by Proposition~\ref{prop:mp2}, with probability one, there exists a subset $B_\infty$ of $M$ and a subsequence $\{B_{n_k}\}$ such that $B_{n_k}$ converges to $B_\infty$ in the $L^1$-metric.
Since $h(\cdot;M)$ is lower-semi-continuous by Proposition~\ref{prop:mp1}, with probability one, $\liminf_{n\to\infty} h(B_n;M) \geq h(B_\infty;M)$.
Since we also have $\liminf_{n\to\infty} h(B_n;M) = H(M)$ a.s., it follows that $h(B_\infty;M) = H(M)$ a.s. and so $B_\infty$ is a Cheeger set of $M$.

Moreover, since $f_n$ leaves $M_{s_n}$ unchanged,
$$\vol_d(A_n\Delta B_n) \leq \vol_d(M\backslash M_{s_n})\to 0\quad\text{as $n\to\infty$}.$$
Hence with probability one, $\mathbf{1}_{A_n} - \mathbf{1}_{B_n} \to 0$ in $L^1$.
Consequently, the sequences $\{A_n\}$ and $\{B_n\}$ have the same accumulation points, and so any convergent subsequence of $\{A_n\}$ converges to a Cheeger set of $M$.

\subsection{Proof of Theorem~\ref{theo:discretemeasure}}
\label{sub:proof_3}
Let $A_{n} = R_{n}\cap M$ and assume, without loss of generality, that $A_n \to A_\infty$ in $L^1$. 
For all $n\geq 1$, and all $f$ in the class of bounded and continuous functions on $\dM$, say $\cC_b(\dM)$, we have
$$
\left| Q_nf - \int_Mf(x)\mathbf{1}_{R_n}(x)\mu({\rm d}x) \right| \leq \sup_{R \in\cR_n} \left|P_n\left(f\mathbf{1}_R\right) - \mu\left(f\mathbf{1}_R\right)\right|,
$$
where $P_n$ is the empirical measure of the sample $X_1, \dots, X_n$.
Using the bound on the covering numbers in Lemma~\ref{lem:vcan}, it is a classical exercise to prove that the collection of functions $x\mapsto f(x)\mathbf{1}_R(x)$ where $R$ ranges over $\cR_n$ is a Glivenko-Cantelli class, whence
\[
\left| Q_nf - \int_Mf(x)\mathbf{1}_{R_n}(x)\mu({\rm d}x) \right| \to 0\quad\text{a.s. as $n\to\infty$}.
\]
Next,
$$\left|  \int_Mf(x)\mathbf{1}_{R_{n}}(x)\mu({\rm d}x) - Qf \right| = \left|  \int_Mf(x)\mathbf{1}_{A_{n}}(x)\mu({\rm d}x) - Qf \right| \leq \|f\|_\infty \mu\left(A_{n}\Delta A_\infty\right),$$
which tends to 0 by definition of $A_\infty$. 
Thus, we have shown that, for all $f$ in $\cC_b(M)$, $\mathbb{P}\left(Q_nf \to Qf\right) = 1$.
Using the separability of $\cC_b(M)$~\cite[p. 131]{doob}, we deduce that
$$\mathbb{P}\big[\forall f\in\cC_b(M),\, Q_nf \to Qf\big]=1,$$
so that the event ``$Q_n$ converge weakly to $Q$'' is of probability 1.











\section*{Acknowledgments}
Ery Arias-Castro was partially supported by a grant from the US National Science Foundation (DMS-06-03890) and wish to thank Lei Ni for stimulating discussions.
Bruno Pelletier and Pierre Pudlo were supported by the French National Research Agency (ANR) under grant ANR-09-BLAN-0051-01.
Bruno Pelletier also thanks Michel Pierre for discussions related to Cheeger sets.

\bibliographystyle{abbrv}
\bibliography{cheeger}

\end{document}